\documentclass[a4paper,11pt, reqno]{amsart}
\usepackage[latin1]{inputenc}

\usepackage{amssymb}
\usepackage{amsmath}
\usepackage{amsthm}
\usepackage{graphicx}
\usepackage{a4wide}
\usepackage{geometry}
\usepackage{amsmath}
\usepackage{amsfonts}
\usepackage{amssymb}
\usepackage{graphicx}%
\providecommand{\U}[1]{\protect\rule{.1in}{.1in}}

\renewcommand{\ll}{{\ell}}

\newcommand{\be}[1]{\begin{equation}\label{#1}}
\newcommand{\ee}{\end{equation}}

\let\pa\partial
\let\r\rho

\newtheorem{thm}{Theorem}
\newtheorem{prop}{Proposition}
\newtheorem{rem}{Remark}
\newtheorem{cor}{Corollary}
\newtheorem{assumption}{Assumption}

\newcommand{\beq}{\begin{eqnarray}}
\newcommand{\eeq}{\end{eqnarray}}
\newcommand{\beqs}{\begin{eqnarray*}}
\newcommand{\eeqs}{\end{eqnarray*}}
\newcommand{\bequ}{\begin{equation}}
\newcommand{\eequ}{\end{equation}}

\def\r{\rho}

\def\T{{\cal T}}
\def\M{{\cal M}}
\newcommand{\cal}{\mathcal}

\def\e{\varepsilon}

\newtheorem{lemma}{Lemma}[]

\let\pa\partial

\usepackage{color}
\newcommand{\corr}[1]{\textcolor{black}{#1}}

\makeatother

\begin{document}
\title[]{Global well-posedness for the primitive equations coupled to
nonlinear moisture dynamics with phase changes}

\author{Sabine Hittmeir}
\address[S. Hittmeir]{Fakult\"at f\"ur Mathematik, Universit\"at Wien, Oskar-Morgenstern-Platz 1, 1090 Wien, Austria}
\email{sabine.hittmeir@univie.ac.at}
\author{Rupert Klein}
\address[R. Klein]{FB Mathematik \& Informatik, Freie Universit\"at Berlin, Arnimallee 6, 14195 Berlin, Germany}
\email{rupert.klein@math.fu-berlin.de}
\author{Jinkai Li}
\address[J. Li]{South China Research Center for Applied Mathematics and Interdisciplinary Studies, South China Normal University, Guangzhou 510631, China.}
\email{jklimath@m.scnu.edu.cn; jklimath@gmail.com}
\author{Edriss S. Titi}
\address[E. S. Titi]{Department of Mathematics, Texas A{\&}M University, College Station,  TX 77840, USA.  Department of Applied Mathematics and Theoretical Physics,
University of Cambridge, Cambridge CB3 0WA, UK. Department of Computer Science and Applied Mathematics, Weizmann Institute of Science,
Rehovot 76100, Israel.} \email{titi@math.tamu.edu; Edriss.Titi@damtp.cam.ac.uk}

.

\maketitle

\begin{center}
  {March 20, 2020}
\end{center}

\begin{abstract}
In this work we study the global solvability of the primitive
equations for the atmosphere coupled to moisture dynamics with phase
changes for warm clouds, where water is present in the form of
water vapor and in the liquid state as cloud water and rain
water. This moisture model contains closures for the phase changes
condensation and evaporation, as well as the processes of
autoconversion of cloud water into rainwater and the collection of
cloud water by the falling rain droplets. It has been used by Klein
and Majda in \cite{KM} and corresponds to a basic form of the bulk
microphysics closure in the spirit of Kessler \cite{Ke} and Grabowski
and Smolarkiewicz \cite{GS}. The moisture balances are strongly
coupled to the thermodynamic equation via the latent heat associated to the
phase changes. In \cite{HKLT} we assumed  the
velocity field to be given and proved rigorously the global existence
and uniqueness of uniformly bounded solutions of the moisture balances
coupled to the thermodynamic equation.
In this paper we present the solvability of a full moist
atmospheric flow model, where the moisture model is coupled to the primitive equations of atmospherical dynamics governing the velocity field. For the derivation of a priori estimates for the velocity field we thereby use the ideas of Cao and Titi \cite{CT}, who succeeded in proving the global solvability of the primitive equations.
\end{abstract}

\allowdisplaybreaks
\section{Introduction}
Moisture and precipitation still cause major uncertainties in numerical weather prediction models and it is our aim here to develop further the rigorous analysis of atmospheric flow models. In a preceding paper \cite{HKLT} we assumed the velocity field to be given and studied the moisture model for water vapor, cloud water and rain water  coupled to the  thermodynamic equation through the latent heat in the setting of
Klein and Majda \cite{KM} corresponding to a basic form of a bulk microphysics model in the spirit of Kessler \cite{Ke} and Grabowski and Smolarkiewicz \cite{GS}.  In this work we couple the moisture dynamics to the primitive equations of the atmospheric dynamics by taking over the ideas of Cao and Titi \cite{CT} for their recent breakthrough on the global solvability of the latter system. Moreover, cases of  partial viscosities and diffusions, arising from the asymptotical analysis in \cite{KM}, will be analyzed in a future work capitalizing on the results by Cao et al. \cite{CLT,CLT1,CLT2}.

A study of a moisture model coupled to the primitive equations has already been carried out by Coti Zelati et al. in \cite{CZH}.  The moisture model there consists of one moisture quantity coupled to temperature and contains only the process of condensation
during upward motion, see e.g. \cite{HW}. Since the source term there is modeled via a Heavy side
function as a switching term between saturated and undersaturated regions, the analysis, however, requires elaborate techniques. Coti Zelati et al. in \cite{BCZ,CZF,CZT} therefore used an approach based on differential inclusions and variational techniques, which have then been coupled to the primitive equations in \cite{CZH}.

\corr{The moisture model we are analyzing here is physically more refined and consists of three moisture quantities for  water vapor, cloud water and rain water. It contains besides the phase changes condensation and evaporation also the autoconversion of cloud water to rain water after a certain threshold is reached, as well as  the collection of cloud water by the falling rain droplets.}

In the remainder of the introduction we first state the moisture model  in pressure coordinates, which have the advantage that under the assumption of hydrostatic balance the continuity equation takes the form of the incompressibility condition.

\subsection{Governing equations}
Solvability of the full geophysical governing equations (without moisture) is a long standing problem. Assuming hydrostatic balance
\beq
\label{hydbal}
\frac{\pa p}{\pa z} = - g \r \,,
\eeq
where  $g$ denotes the graviational acceleration, the equations reduce to the well-known primitive equations and only recently the global well-posedness of strong solutions could be proven by Cao and Titi \cite{CT} for the incompressible ocean dynamics.  The density of air $\r$ in the atmosphere in comparison to the ocean  however varies strongly with height, and the incompressibility assumption is only justified when describing shallow phenomena. Thus for the atmosphere in general the full compressible governing equations need to be considered. However, under the assumption of hydrostatic balance \eqref{hydbal}, which in particular guarantees the pressure to decrease monotonically with height, the pressure $p$ can be used as the vertical coordinate.
Switching to the pressure coordinates $(x,y,p)$ has the main advantage that the continuity equation takes the form of  the  incompressibility condition
\beq\label{div}
\pa_x u + \pa_y v + \pa_p \omega=0\,\qquad \textnormal{where} \quad \omega =\frac{dp}{dt}\,,\eeq
see, e.g., Lions et al.\,\cite{LTW} and  Petcu et al.\,\cite{PTZ}. Thus the ideas of Cao and Titi \cite{CT} can be taken over for the atmospheric primitive equations in pressure coordinates, as we shall also see below and we therefore work in the following with the governing equations in the pressure coordinates and use hereafter the notation
\beqs
 \mathbf{u}=(u,v)\,,\qquad &\nabla_h=(\pa_x,\pa_y)\,,\qquad & \Delta_h=\pa_x^2+\pa_y^2\,.
\eeqs
We note that the vertical velocity $\omega$ in pressure coordinates takes the opposite sign as the vertical velocity in the cartesian coordinates, i.e., $\omega<0$ for upward motion and $\omega>0$ for downward motion. Moreover the  derivatives as well as the  velocity components have different units in vertical and horizontal directions. Nevertheless the total derivative  in pressure coordinates reads
\beq\label{tot.p}
\frac{d}{dt}&=&\pa_t + \mathbf{u}\cdot \nabla_h + \omega\pa_p \,.
\eeq
For the eddy viscosity closure of turbulence and molecular transport we use
\beq\label{tot.diff}
{\cal D}^* &=& \mu_*\Delta_h+ \nu_*\pa_p\left(\left(\frac{g p }{R \bar T}\right)^2\pa_p\right),
\eeq
where $\bar T=\bar T(p)$ corresponds to some background distribution of the temperature,
being uniformly bounded from above and below, and $R$ is the individual
gas constant. The operator ${\cal D}^*$  thereby provides a
close approximation to the full Laplacian in cartesian coordinates, see
also \cite{LTW,PTZ}.
The governing equations in  pressure coordinates $(x,y,p)$ with corresponding velocities $(\mathbf{u},\omega)$ become
\beq
&&\pa_t\mathbf{u}+(\mathbf{u}\cdot \nabla_h)\mathbf{u} + \omega \pa_p\mathbf{u} + \nabla_h \Phi +f (\mathbf{k}\times\mathbf{v})_h={\cal D}^{\mathbf{u}}\mathbf{u},\label{equ.1}\\
&&\pa_p \Phi + \frac{R T}{p}=0,\label{equ.2}\\
&&\label{inc.v} \nabla_h\cdot\mathbf{u}+\pa_p\omega =0,\label{equ.3}
\eeq
where $\Phi$ denotes the geopotential $\pa_z \Phi =g$, and the second equation combines the ideal gas law (see \eqref{ideal} below) and the hydrostatic balance
equation \eqref{hydbal}. Here $\mathbf{v}:=(\mathbf{u},\omega)$, $\mathbf{k}=(0,0,1)$, and $(\mathbf{k}\times\mathbf{v})_h$ are the first two components of $\mathbf{k}\times\mathbf{v}$.
Now the density $\r$ does not appear in the system anymore, which is connected to the other thermodynamic quanitities via
 the ideal gas law
\beq\label{ideal}
p=R \r T\,.
\eeq
A typical measure for quantification of moisture are mixing ratios, which compare the density of the moisture component to the density of dry air (denoted as $\r_d$). We assume to be in the warm cloud regime, where no ice and snow phases occur and water is therefore present in the form of water vapor (with density $\r_v$) and in the liquid state as cloud water and rain water (with corresponding densities $\r_c$ and $\r_r$), such that we have the mixing ratios
\begin{equation}\label{qv}
q_{v}=\frac{\r_v}{\r_d}\,, \qquad  q_{c}=\frac{\r_c}{\r_d}\,,\qquad q_{r}=\frac{\r_r}{\r_d}\,.
 \end{equation}
For these mixing ratios we then have the  moisture balances
\beq
\label{eq.qv}\frac{d q_v}{dt}  &=& S_{ev} - S_{cd}+\cal{D}^{q_v} q_v\,,\\
\label{eq.qc}\frac{d q_c}{dt} &=& S_{cd} - S_{ac}- S_{cr}+\cal{D}^{q_c} q_c\,,\\
\label{eq.qr}\frac{d q_r}{dt}+V\pa_p\left(\frac{p}{R_d \bar{T}} q_r\right) &=& S_{ac} + S_{cr}- S_{ev}+\cal{D}^{q_r} q_r\,,
\eeq
where the total derivative is given according to \eqref{tot.p} and the diffusion terms are as in \eqref{tot.diff}.
The source terms  $S_{ev}, S_{cd}, S_{ac}, S_{cr}$ are, respectively, the rates of evaporation of rain water, the condensation of water vapor to cloud water and
the inverse evaporation process, the auto-conversion of cloud water into rainwater by accumulation of microscopic droplets,
and the collection of cloud water by falling rain. Moreover $V$ denotes the terminal velocity of falling rain and is assumed to be constant.

The thermodynamic equation accounts for the diabatic source and sink terms, such as latent heating, radiation effects,
but we will in the following only focus on the effect of latent heat in association with phase
changes (see, e.g., \cite{CZF,CZH,HKLT,KM}).
The temperature equation in pressure coordinates then reads, see, e.g., \cite{C,HK},
\beq
\label{eq.T}\frac{dT}{dt}-\frac{R}{c_p}\frac{T}{p} \omega =\frac{L}{c_p}(S_{cd}-S_{ev})+ {\cal D}^TT\,,
\eeq
where the heat capacity $c_p$ and the latent heat $L$ are assumed to be constant.

To describe the state of the atmosphere a common thermodynamic quantity used instead of the temperature is the potential temperature
\begin{equation}\label{PT}
\theta = T \left(\frac{p_0}{p}\right)^\kappa\,, \qquad \textnormal{where} \qquad \kappa = \frac{ R}{c_{p}}\,.
\end{equation}
The potential temperature has the main advantage that the left-hand side of \eqref{eq.T} simply reduces to
$\frac{T}{\theta}\frac{d}{dt}\theta$. This property was essential in the
preceding works \cite{CZT} and \cite{HKLT} to derive a priori
nonnegativity and boundedness of the temperature and moisture
components.


\begin{rem}
In the present model the difference of the gas constants for dry air and water vapor as well as the  dependence of the internal energy on the moisture components via different heat capacities  is neglected. These additional terms that would arise in a more precise thermodynamical setting are small in principle and therefore often not taken into account.
\corr{It has been, however, revealed in \cite{HK} that, e.g., in the presence of deep convective clouds the refined thermodynamical setting can be essential. The according moisture model has a much stronger coupling of the thermodynamic equation to the moisture components and will be investigated in a forthcoming paper.}
\end{rem}

\subsection{Explicit expressions for the source terms}
The threshold for phase changes is saturation, which is defined via the saturation mixing ratio $q_{vs}$. Saturation thereby is reached when $q_v=q_{vs}$, whereas the air is undersaturated if $q_v<q_{vs}$ and oversaturated if $q_v>q_{vs}$, respectively.
For the given function $q_{vs}$ we pose the natural assumption to depend continuously on $p$ and $T$ and to vanish below and above some critical temperatures (given in Kelvin), i.e.
\beq
\label{cond.qvs}
\label{nonneg.qvs}
 q_{vs}(p,T)=0 \quad \textnormal{for} \quad  T \leq T_{A} \quad \textnormal{and} \quad T\geq T_{B} \,,
\eeq
for some $0\leq T_A\leq  T_B\,$, which is helpful for proving nonnegativity of the solution. Moreover we assume $q_{vs}$ to be nonnegative, uniformly bounded and to be Lipschitz continuous with respect to $T$, i.e., we assume
\beq\label{LC}
|q_{vs}(p,T_1)-q_{vs}(p,T_2)|\leq C|T_1-T_2|,
\eeq
for a positive constant $C$.  This constant actually depends on the pressure as it grows approximately inversely proportional to $p$. We are however only interested in the lower part of the atmosphere, where weather related phenomena are taking place. There the pressure is uniformly bounded from below (e.g. by 100\,hPa), such that the constant $C$ in \eqref{LC} can be assumed to be positive and uniformly bounded.  For more details we refer also to \cite{HKLT}.

For the source terms of the mixing ratios we take over the setting of Klein and Majda \cite{KM} corresponding to a basic form of
the bulk microphysics closure in the spirit of Kessler \cite{K} and Grabowski and
Smolarkiewicz \cite{GS}, which has also been used in the preceding work \cite{HKLT}:
\beq
S_{ev}&=&C_{ev}T (q_r^+)^\beta(q_{vs}-q_v)^+\,,\qquad \beta \in (0,1],\label{Sev}\\
S_{cr}&=&C_{cr} q_c q_r,\qquad \\
S_{ac}&=&C_{ac} (q_c-q_{ac}^*)^+, \\
\label{Scd}
S_{cd}&=&C_{cd}(q_v-q_{vs})q_c  +C_{cn}(q_v-q_{vs})^+\,,
\eeq
where $C_{ev},C_{cr},C_{ac},C_{cd},C_{cn}$ are dimensionless rate constants. Moreover, $(g)^+=\max\{0,g\}$
and $q_{ac}^*$ denotes the threshold for cloud water mixing ratio beyond which autoconversion of cloud water into precipitation becomes active.

Exponents $\beta \in (0,1)$ cause difficulties in the analysis, in particular for the uniqueness of the solutions. This problem however was overcome in \cite{HKLT} by introducing new unknowns, which allow for certain cancellation properties of the source terms and reveal advantageous monotonicity properties. \corr{This procedure in particular relies on the fact, that the evaporation constitutes a sink in the equation for temperature, i.e. for positive temperatures $T$ the term $S_{ev}$ is nonnegative and arises with a negative sign in the thermodynamic equation.}

The rest of this paper is arranged as follows. In section
\ref{sec.main}, we formulate the full problem with boundary
conditions and state the main results on the global existence and
uniqueness of solutions. In section \ref{sec.apx},
we prove the existence
and uniqueness, and the uniform a priori estimates of solutions to an
approximate system, which is nothing but the original one by
replacing $S_{ev}$ (which may be only H\"older continuous in $q_r$)
by $S_{ev,\varepsilon}$ \corr{(see \eqref{Sev.eps} below)} which is Lipschitz in $q_r$. In section
\ref{sec.pf}, based on the results in section \ref{sec.apx}, we give the
proof of the global existence result, and the uniqueness is also
established by using the idea in our previous work \cite{HKLT}.

Throughout this paper, unless explicitly specified, we use $C$
to denote a generic positive constant depending only on the given
functions in the boundary conditions, the initial data, and the physical
parameters appearing in the original system (but not on the parameter
$\varepsilon$ arising in the approximate system introduced in the next section).

\section{Formulation of the problem and main results}\label{sec.main}

Recall the momentum equation
\begin{equation}\label{eq.u}
\pa_t\mathbf{u}+(\mathbf{u}\cdot \nabla_h) \mathbf{u} +\omega \pa_p \mathbf{u} +  f (\mathbf{k}\times\mathbf{v})_h+\nabla_h \Phi ={\cal D}^{\mathbf{u}}\mathbf{u},
\end{equation}
where $\mathbf{v}=(\mathbf{u},\omega)$, $\mathbf{k}=(0,0,1)$,
and $(\mathbf{k}\times\mathbf{v})_h$ are the first two components of $\mathbf{k}\times\mathbf{v}$. According to the incompressibility condition (\ref{inc.v}) and the boundary condition (\ref{bound.0}), we have a diagnostic equation for $\omega$
\begin{equation}
\label{eq.omega}
\omega(t,x,y,p)=\int_p^{p_0}\nabla_h\cdot \mathbf{u}(t,x,y,s) ds\,.
\end{equation}
By the hydrostatic balance (\ref{equ.2}), we have
\begin{equation}
\Phi(t,x,y,p)=\Phi_s(x,y,t) + \int_p^{p_0} \frac{R}{\sigma}T(t,x,y,\sigma) d\sigma.\label{eq.Phi}
\end{equation}

\corr{For analyzing cloudy air phenomena either numerically or analytically a bounded domain of cylindrical form is a natural choice, see also \cite{CZF,HKLT}, and we thus consider a domain ${\M}$ defined as}
\beqs
{\M}=\{(x,y,p)\,|\, (x,y) \in {\M}', p\in (p_1,p_0)\}\,,
\eeqs
where ${\M}'$ is a smooth bounded domain in $\mathbb{R}^2$, and $0<p_1<p_0$. The boundary is given by
\beqs
&&\Gamma_0=\M'\times\{p_0\},\quad \Gamma_1=\M'\times\{p_1\},\quad \Gamma_\ell=\partial\mathcal M'\times(p_1,p_0).
\eeqs
The boundary conditions are:
\beq
\label{bound.0}&\Gamma_0:\ & \pa_p\mathbf{u}=-\alpha_{\mathbf{u}}\mathbf{u}\,,\quad \omega=0\,,\quad \pa_pT=\alpha_{0T}(T_{b 0}-T)\,,\qquad \nonumber\\
&&  \pa_pq_{j}=\alpha_{0j} (q_{b 0 j}-q_j) \,, \quad \text{for }j\in\{v,c,r\}\,,\qquad\\
&\Gamma_1:\ & \pa_p\mathbf{u}=0\,,\quad  \omega=\pa_pT=\pa_pq_j=0\,,\quad j\in \{v,c,r\}\,, \label{bound.1}\\
\label{bound.ll}&\Gamma_{\ll}:\ &\mathbf{u}\cdot \mathbf{n}=0\,,\quad \partial_n\textbf{u}\times\textbf{n}=0\,,\quad \pa_{\mathbf{n}}T=\alpha_{{\ll}T}(T_{b \ll}-T)\,,\nonumber\\
&&\pa_nq_{j}=\alpha_{{\ll} j} (q_{b {\ll} j}-q_j) \,, \quad \text{for }j\in\{v,c,r\},
\eeq
where $\alpha_{0 j }, \alpha_{\ll j}, \alpha_{0 T }, \alpha_{\ll T}$ are
given nonnegative constants, and $T_{b 0 }, T_{b \ll},  q_{b 0 j },
q_{b \ll j }$, which can depend on time, are given nonnegative and sufficiently smooth functions. Here
$\mathbf{n}$ denotes the outward normal direction on $\partial\mathcal
M'$. Note that the boundary conditions (\ref{bound.0})--(\ref{bound.ll})
reduce to those in \cite{CZH} if $\alpha_{\ll T}$, $\alpha_{{\ll} v}$,
$\alpha_{{\ll} c}$, and $\alpha_{{\ll} r}$ are set tozero.
The initial condition is
\begin{equation}
  \label{IC}
(\mathbf{u}, T, q_v, q_c, q_r)|_{t=0}=(\mathbf{u}_0, T_0, q_{v0}, q_{c0}, q_{r0}).
\end{equation}


Throughout this paper, we use the abbreviation
\[\|f\|=\|f\|_{L^2(\M)}\,,\qquad \|f\|_{L^p}=\|f\|_{L^p(\M)}\,.\]
According to the weight in the vertical diffusion terms, we  introduce the weighted norm
\[\|f\|_w=\Big\|\left(\frac{g p}{R_d\bar T}\right) f\Big\|\,,\]
which, since the weight $\frac{gp}{R_d\bar T}$ is uniformly bounded from above and below by positive constants, is equivalent to the $L^2$-norm.
Moreover, we shall often use for convenience the notation
\[\|(f_1,\dots,f_n)\|^2=\sum_{j=1}^n\|f_j\|^2\]
We state our main result on the global existence and uniqueness of solutions to the fully coupled  system in the following:

\begin{thm}
\label{thm}
Assume that $\mathbf{u}_0, T_0, q_{v0}, q_{c0}, q_{r0}\in H^1
(\mathcal M)$ and $T_0,q_{v0}, q_{c0}, q_{r0}\in L^\infty(\mathcal M)$, with $T_0$, $q_{v0}$, $q_{c0}$,
$q_{r0}\geq0$ in $\mathcal M$ and
$\int_{p_0}^{p_1}\nabla_h\cdot\mathbf{u}_0dp=0$ on $\mathcal M'$. Then, system \eqref{eq.qv}--\eqref{eq.T}, \eqref{eq.u}--\eqref{eq.Phi},
subject to \eqref{bound.0}--\eqref{IC}, has a unique global
in time solution $(\mathbf{u}, T, q_v, q_c, q_r)$, satisfying
\begin{eqnarray*}
&T, q_v, q_c, q_r\geq0 \quad\textnormal{and} \quad  T,q_v,q_c,q_r\in L^\infty(0,\mathcal T; L^\infty(\M)),\\
&\mathbf{u}, T, q_v, q_c, q_r\in C([0,\mathcal T]; H^1(\mathcal M))\cap L^2(0,\mathcal T; H^2(\mathcal M)),
\\
&\partial_t\mathbf{u},\partial_tT, \partial_tq_v, \partial_tq_c, \partial_tq_r\in L^2(0,\mathcal
T; L^2(\mathcal M)),
\end{eqnarray*}
for any $\mathcal T\in(0,\infty)$.
\end{thm}

\corr{Some comments concerning
the proof of Theorem \ref{thm} are given in order. For the global existence of solutions, the key is to
get appropriate a priori estimates for the solutions, which are mainly obtained by adopting the ideas from \cite{HKLT} for the moisture model with given velocity field and the work of Cao and Titi \cite{CT}, who proved global well-posedness for the primitive equations for the ocean. In particular, we use as in \cite{CT} the ideas of decomposing the velocity into the barotropic (vertically averaged) and the baroclinic (the according deviation) components and using the Ladyzhenskaya type inequality (see Lemma \ref{lemlad} in the Appendix) to derive the $L^\infty(0,\mathcal T; L^6(\mathcal M))$ of the horizontal velocity, see the proof of Proposition \ref{est.u.pu}, in the below. Besides, similar to \cite{CT}, due to the anisotropic property of the system (\ref{eq.u})--(\ref{eq.Phi}), we also use anisotropic treatments to the horizontal derivatives and vertical derivatives, that is the a priori estimates for the vertical derivatives are carried out separately before working on the horizontal ones. Furthermore, as already mentioned before, we use the idea of introducing new
unknowns and making use of cancellations as in \cite{HKLT} to overcome the difficulty cased by the exponent $\beta\in(0,1)$ in the source term $S_{ev}$ to prove the uniqueness. Yet there are some technical differences of arguments in the current paper comparing with
those in \cite{CT} and \cite{HKLT}, due to the presence of the antidissipative term $\frac{R}{c_p}\frac{T}{p}\omega$ and the latent heating term $\frac{L}{c_p}(S_{cd}-S_{ev,\varepsilon})$ in the thermodynamic equation (\ref{eq.T}), as well as the coupling of the momentum equations to the moisture system through the transport terms. The antidissipative term makes the $L^1$ type estimate for $T$, i.e., Proposition \ref{basic}, in the below,
become a necessary step to get the further estimates, while
the latent heat provides a strong coupling to the moisture dynamics, which are in turn transported with the air velocity. These additional challenges due to the strong coupling of all solution components are overcome by careful derivations of a priori estimates, which need to be elaborated in the right order, as explained more detailed below. Moreover, since the velocity field is no longer assumed to be a given function as in \cite{HKLT}, the required $L^\infty(0,\mathcal T; H^1(\mathcal M))$ of the moisture components
cannot be derived by applying the parabolic estimates to the moisture system (\ref{eq.qv})--(\ref{eq.qr}).
}

\section{An approximated system: existence and a priori estimates}
\label{sec.apx}

This section is devoted to proving the global existence, uniqueness, and uniform a priori estimates of the following $\e$--approximate system to (\ref{eq.qv})--(\ref{eq.qr}), (\ref{eq.T}), and (\ref{eq.u})--(\ref{eq.Phi}),
\begin{eqnarray}
&&\pa_t\mathbf{u}+(\mathbf{u}\cdot \nabla_h)\mathbf{u} + \omega
\pa_p\mathbf{u} + \nabla_h \Phi +f (\mathbf{k}\times\mathbf{v})_h={\cal
D}^{\mathbf{u}}\mathbf{u},\label{Au.1}\\
&&\pa_p \Phi + \frac{R T}{p}=0\label{Au.2}\\
&&\nabla_h\cdot\mathbf{u}+\pa_p\omega =0,\label{Au.3}\\
&&\pa_tT+(\mathbf{u}\cdot\nabla_h)T+\omega\partial_pT
-\frac{R
}{c_p}\frac{T}{p} \omega=\frac{L}{c_p}(S_{cd}^+-S_{ev,\varepsilon}^+)+ {\cal D}^TT,\label{AT}\\
&&\partial_tq_v+\mathbf{u}\cdot\nabla_hq_v+\omega\partial_pq_v  =
S_{ev,\varepsilon}^+
- S_{cd}^++\cal{D}^{q_v} q_v\,,\label{Aqv}\\
&&\partial_tq_c+\mathbf{u}\cdot\nabla_hq_c+\omega\partial_pq_c= S_{cd}^+ -S_{ac}^+- S_{cr}^++\cal{D}^{q_c} q_c\,,\label{Aqc}\\
&&\partial_tq_r+\mathbf{u}\cdot\nabla_hq_r+\omega\partial_pq_r+V\pa_p
\left(\frac{p}{R_d\bar{T}} q_r\right) = S_{ac}^+ + S_{cr}^+-
S_{ev,\varepsilon}^++\cal{D}^{q_r} q_r\,,\label{Aqr}
\end{eqnarray}
where $\varepsilon\in(0,1)$ is  fixed  and
\begin{eqnarray}
\label{Sev.eps}&&S_{ev,\varepsilon}^+=C_{ev}  R T^+q_r^+(q_r^++\varepsilon)^{\beta-1}(q_{vs}(p,T)-q_v)^+,\quad \varepsilon\in(0,1),\\
&&S_{cr}^+=C_{cr}q_c^+q_r^+,\qquad S_{ac}^+=S_{ac}=C_{ac}(q_c-q_{ac}^*)^+,\\
&&S_{cd}^+=C_{cd}(q_v^+-q_{vs}(p,T))q_c^++C_{cn}(q_v-q_{vs}(p,T))^+.
\end{eqnarray}
Unlike the original $S_{ev}$, the corresponding approximation
$S_{ev,\varepsilon}^+$ is Lipschitz with respect to $q_r$, and it approximates
$S_{ev}$ as $\varepsilon$ tends to zero.

Since all the nonlinear terms $S_{ev,\varepsilon}^+, S_{cr}^+,
S_{ac}^+,$ and $S_{cd}^+$ are Lipschitz with respect to $q_v, q_c, q_r,$ and $T$, the local, in time, existence and
uniqueness of strong solutions to the initial boundary value problem of the $\e$--approximate system (\ref{Au.1})--(\ref{Aqr}) follows the standard contraction mapping fixed point principle and  we obtain the  following proposition on the local, in time, existence and uniqueness result.

\begin{prop}
  \label{Aloc}
Assume that $\mathbf{u}_0, T_0, q_{v0}, q_{c0}, q_{r0}\in H^1
(\mathcal M)$ and $q_{v0}, q_{c0}, q_{r0}\in L^\infty(\mathcal M)$, with $T_0$, $q_{v0}$, $q_{c0}$,
$q_{r0}\geq0$ on $\mathcal M$, and
$\int_{p_0}^{p_1}\nabla_h\cdot\mathbf{u}_0dp=0$ on $\mathcal M'$. Then,
there is a positive time $\mathcal T_0$ depending only on
the upper bound of $\|(\mathbf{u}_0, T_0, q_{v0}, q_{c0}, q_{r0})\|_{H^1(\mathcal M)}$, which is independent of $\e$,
such that system (\ref{Au.1})--(\ref{Aqr}), subject to
(\ref{bound.0})--(\ref{IC}), has a unique strong solution
$\mathbf{u}, T, q_v, q_c, q_r$
on $\mathcal M\times(0,\mathcal T_0)$, satisfying
\begin{eqnarray*}
  &&\mathbf{u}, T, q_v, q_c, q_r\in C([0,\mathcal T_0]; H^1(\mathcal M))\cap L^2(0,\mathcal T_0; H^2(\mathcal M)), \quad
\\
&&\partial_t\mathbf{u},\partial_tT, \partial_tq_v, \partial_tq_c, \partial_tq_r\in L^2(0,\mathcal T_0; L^2(\mathcal M)).
\end{eqnarray*}
\end{prop}

In the following we let $(\mathbf{u}, T, q_v, q_c, q_r)$ be the solution obtained in
Proposition \ref{Aloc}, and extend it
to the maximal interval of existence $(0,\mathcal T_\text{max})$, where $\mathcal T_\text{max}$ is characterized
as
\begin{equation}\label{T_max}
  \limsup_{\mathcal T\rightarrow\mathcal T_{\text{max}}^-}\|(\mathbf{u}, T, q_v, q_c, q_r)\|_{H^1(\mathcal M)}=\infty,\quad \mbox{if } \mathcal T_{\text{max}}<\infty.
\end{equation}

The following assumption will be made in the subsequent propositions throughout this section.
\begin{assumption}\label{ass} Let all the assumptions in Proposition \ref{Aloc} hold, and let the unique solution $(\mathbf{u}, T, q_v, q_c, q_r)$ obtained in Proposition \ref{Aloc} be extended in the above way to the maximal time of existence $\mathcal T_{\text{max}}$.
\end{assumption}
The aim of the rest of this section is to show that $\mathcal T_{\text{max}}=\infty$, and to establish the a priori estimates that are independent of $\varepsilon\in(0,1)$. \corr{Here at first nonnegativity and uniform boundedness of $(T,q_v, q_c, q_r)$ are proven. This in particular allows to derive estimates for the moisture mixing ratios  in $L^\infty(0,\T;L^2(\M)) \cap L^2(0,\T;H^1(\M))$, which grow continuously in $\T$. The energy of the horizontal velocity is combined with the integral of the temperature, since this allows for the cancellation of the geopotential term involving the vertical velocity. This bound on the horizontal velocity then allows to obtain an $L^\infty(0,\T;L^2(\M)) \cap L^2((0,\T;H^1(\M))$ control of $T$. Due to the strong nonlinearty of the system the direct derivation of a priori estimates for the gradients of the solution components is not possible and we employ the main idea of Cao and Titi \cite{CT} by first bounding $\mathbf{u}$ in $L^\infty(0,\T;L^6(\M))$. The key idea here in \cite{CT} is to split the solution into the barotropic (vertically averaged) and the baroclinic component (the according deviation). Since the geopotential due to the hydrostatic assumption enters the equations in fact only as a two-dimensional surface, it is absent in the dynamics of the baroclinic mode, allowing to close the estimate for  $\mathbf{u}$ in $L^\infty(0,\T;L^6(\M))$. Based upon this estimate then further an a priori estimate for $\pa_p \mathbf{u}$ and thereafter $\pa_p q_{j}$ can be derived. These and previous bounds allow further to control the horizontal gradients $\nabla_h \mathbf{u}$ and subsequently also $\nabla_h q_j$. Finally also a control of the gradient of $T$ is derived, which completes the set of estimates allowing to bound all solution components in $L^\infty(0,\T;H^1(\M))\cap L^2(0,\T;H^2(\M))$.}

\corr{As a first step we derive the nonnegativity and uniform boundedness of $(T,q_v, q_c, q_r)$ in the next proposition in a similar fashion to \cite{HKLT}:}

\begin{prop}\label{bddq}
Let Assumption \ref{ass} hold, then the solution $(T,q_{v},q_{c},q_{r})$ satisfies
\beq
0\leq q_{v}\leq  q_v^*\,,\quad 0\leq q_{c}\leq q_c^*\,,\quad 0\leq q_{r}\leq  q_r^*\,,\quad 0\leq T\leq T^*\,, \qquad \textnormal{on} \ \mathcal M \times (0,\mathcal T),
\eeq
for any $\mathcal T \in (0,\mathcal T_{max})$, where
\beq\label{qv.star}
q_v^*=\max\big\{\|q_{v 0}\|_{L^\infty{(\M)}},\|q_{b0v}\|_{L^\infty((0,\T)\times\M')},\|q_{b\ll v}\|_{L^\infty((0,\T)\times\Gamma_\ll)}, q_{vs}^*\big\}
\eeq
with $q_{vs}^*=\max q_{vs}$. Moreover $q_c^*,q_r^*,T^*$ are continuous in $\mathcal T$ and depend on the following quantities:
\beq
&&q_c^*\, =C_{q_c}\big(\T, \|q_{c 0}\|_{L^\infty{(\M)}},\|q_{b0c}\|_{L^\infty((0,T)\times\M')},\|q_{b\ll c}\|_{L^\infty((0,T)\times\Gamma_\ll)},q_v^*,q_{vs}^*\big)\,,\\
&&q_r^*\, =C_{q_r}\big(\T, \|q_{r 0}\|_{L^\infty{(\M)}},\|q_{b0r}\|_{L^\infty((0,T)\times\M')},\|q_{b\ll r}\|_{L^\infty((0,T)\times\Gamma_\ll)},q_c^*\big)\,,\\
&&T^*=C_{T}\big(\T,\|T_0\|_{L^\infty{(\M)}},\|T_{b0}\|_{L^\infty((0,T)\times\M')}, \|T_{b\ll}\|_{L^\infty((0,T)\times\Gamma_\ll)},q_v^*,q_c^*,q_{vs}^*\big)\,.
\eeq
\end{prop}

\begin{proof}
We only state the key steps of the proof and refer to the proof of Proposition 3.2 in \cite{HKLT} for more details. Due to the vanishing of the antidissipative term in the potential temperature $\theta$, the latter is used instead of the temperature for the derivation of the maximum principle. As a first step then the nonnegativity of $q_c,q_v,q_r$ and $\theta$ are  derived  in exactly this order by employing the Stampacchia method. Therefore  the according equations are multiplied with the corresponding negative parts of the solution components and after integration and integration by parts  the Gronwall inequality is applied to show that if the solution was nonnegative initially it has to remain nonnegative for all times. The same method can then be used to show uniform boundedness of $q_v$ by deriving an $L^2$-estimate for $(q_v-q_v^*)^+$. The boundedness of $q_c, q_r$ and $\theta$ again in this order follow by employing iterative estimations for cutoff functions in $L^m$ and then passing to the limit as $m\rightarrow \infty$ (see Proposition 3.2 in \cite{HKLT}).
\end{proof}

Due to the nonnegativity of $q_v, q_c, q_r,$ and $T$ we obtain also for the source terms
 \[S_{cr}^+=S_{cr}, \quad S_{cd}^+ =S_{cd} \quad \textnormal{and} \quad S_{ev,\varepsilon}^+=S_{ev,\varepsilon}=C_{ev}Tq_r(q_r+\e)^{\beta-1}(q_{vs}-q_v)^+\,,\]
such that in the following we can drop the positive signs in the according notations. Proposition \ref{bddq} in particular also implies the boundedness of the source terms. To see the uniform boundedness of $S_{ev,\varepsilon}$ we notice that since $\beta \in (0,1)$, we have $q_r(q_r+\varepsilon)^{\beta-1}\leq q_rq_r^{\beta-1}=q_r^\beta$, and recall that $q_{vs}(p,T)=0$, for $T\geq T_{B}$ (see (\ref{cond.qvs})). Therefore, one has
\begin{equation}
  \label{bddsource}
0\leq S_{ev,\epsilon},S_{cr}, S_{ac},|S_{cd}|\leq C,
\end{equation}
where the constant $C$, as the upper bounds in Proposition \ref{bddq}, depends on $\mathcal T$, the initial data and the given functions in the boundary conditions (\ref{bound.0})--(\ref{bound.ll}).

\begin{prop}
\label{H1q}
Let Assumption \ref{ass} hold, then  there exists a function $K_0(t)$, which depends on the initial and boundary data, and is continuous for all $t\geq 0$, such that  the estimate
\begin{equation*}
\int_0^{\mathcal T}
\|\nabla(q_v,q_c,q_r)\|^2dt\leq K_0(\mathcal T),
\end{equation*}
holds for any $\mathcal T\in[0,\mathcal T_{\text{max}})$.
\end{prop}
\begin{proof}
The conclusion follows from multiplying (\ref{Aqv}), (\ref{Aqc}), and
(\ref{Aqr}), respectively, with $q_v, q_c,$ and $q_r$,
performing integration, integrating by parts and using the boundary conditions as well as  Young's inequality, \corr{where all} integrals involving the source terms $S_{ev, \varepsilon}, S_{cr},$ $S_{ac}, S_{cd}$ are uniformly bounded, due to the a priori bounds obtained in Proposition \ref{bddq} concluding the proof.
\end{proof}

Furthermore we have the basic energy inequality contained in the following proposition.

\begin{prop}[Basic energy estimate]
\label{basic}
Let Assumption \ref{ass} hold, then there exists a function $K_1(t)$, which depends on the initial and boundary data, and is continuous for all $t\geq 0$, such that
  $$
\sup_{0\leq t\leq\mathcal T}\int_\mathcal M(|\mathbf{u}|^2+T)d\mathcal M+\int_0^{\mathcal T}
(\|\nabla_h \mathbf{u}\|^2+\|\partial_p\mathbf{u}\|_w^2)dt\leq K_1(\mathcal T),
$$
for any $\mathcal T\in[0,\mathcal T_{\text{max}})$.
\end{prop}

\begin{proof}
  Multiplying equation (\ref{Au.1}) by $\mathbf{u}$, integrating the resultant over $\mathcal M$, then it follows from integration by parts that
\begin{equation}
  \frac12\frac{d}{dt}\|\mathbf{u}\|^2-\int_\mathcal M\mathcal D^{\mathbf{u}}\mathbf{u}\cdot\mathbf{u}d\mathcal M=-\int_\mathcal M\nabla_h\Phi\cdot\mathbf{u}d\mathcal M,\label{egyu.1}
\end{equation}
here we have used the facts that $(\mathbf{k}\times\mathbf{v})_h\cdot\mathbf{u} =\mathbf{u}^\perp\cdot\mathbf{u}=0$, and that the integral involving the
convection terms vanishes after integration by parts, due to the incompressibility condition (\ref{Au.3}) and the boundary conditions (\ref{bound.0})--(\ref{bound.ll}). For the integral of the dissipation term, it follows from integration by parts and the boundary conditions (\ref{bound.0})--(\ref{bound.ll}) for $\mathbf{u}$ that
\begin{eqnarray}
-\int_{\mathcal M}\mathcal
D^{\mathbf{u}}\mathbf{u}\cdot\mathbf{u}d\mathcal M&=&-\int_{\mathcal
M}\left[\mu_{\mathbf{u}}\Delta_h\mathbf{u}
+\nu_{\mathbf{u}}\partial_p\left( \left(\frac{gp}{R_d\bar
T}\right)^2\partial_p\mathbf{u}\right)\right] \cdot\mathbf{u}d\mathcal
M\nonumber\\
&=&-\mu_{\mathbf{u}}\int_{\partial\Gamma_{\ell}}\partial_\mathbf{n}
\mathbf{u}\cdot
\mathbf{u}d\Gamma_{\ell}-\left.\nu_{\mathbf{u}}\int_{\mathcal
M'}\frac{gp}{R_d\bar T}\partial_p\mathbf{u}\cdot\mathbf{u}d\mathcal
M'\right|_{p=p_1}^{p_0}\nonumber\\
&&+\int_\mathcal
M\left(\mu_{\mathbf{u}}|\nabla_h\mathbf{u}|^2+\nu_{\mathbf{u}}\left(
\frac{gp}{R_d\bar T}\right)^2|\partial_p\mathbf{u}|^2\right)d\mathcal
M\nonumber\\
&=&\frac{\nu_{\mathbf{u}}gp_0\alpha_{\mathbf{u}}}{R_d\bar
T(p_0)}\left.\int_{\mathcal M'}|\mathbf{u}|^2d\mathcal M'\right|_{p_0}+\mu_{\mathbf{u}}
\|\nabla_h\mathbf{u}\|^2+\nu_{\mathbf{u}}\|\partial_p\mathbf{u}\|_w^2
\nonumber\\
&\geq&\mu_{\mathbf{u}}
\|\nabla_h\mathbf{u}\|^2+\nu_{\mathbf{u}}\|\partial_p\mathbf{u}\|_w^2.
\label{egyu.2}
\end{eqnarray}
For the term $-\int_\mathcal M\nabla_h\Phi\cdot\mathbf{u}d\mathcal M$, it follows from integrating by parts, using the boundary conditions (\ref{bound.0})--(\ref{bound.ll}) for $\mathbf{u}$, and equations (\ref{Au.2})--(\ref{Au.3}),
that
\begin{eqnarray}
  -\int_\mathcal M\nabla_h\Phi\cdot\mathbf{u}d\mathcal M
  &=& -\int_{\Gamma_\ell}\Phi\mathbf{u}\cdot\textbf{n}d\Gamma_\ell
  +\int_\mathcal M\Phi\nabla_h\cdot\mathbf{u}d\mathcal M\nonumber\\
  &=&-\int_\mathcal M\Phi\partial_p\omega d\mathcal M=
\int_\mathcal M\partial_p\Phi\omega d\mathcal M=-\int_\mathcal M\frac{ RT}{p}\omega d\mathcal M.\label{egyu.3}
\end{eqnarray}
Substituting (\ref{egyu.2}) and (\ref{egyu.3}) into (\ref{egyu.1}) yields
\begin{equation}
  \label{egyu}
\frac12\frac{d}{dt}\|\mathbf{u}\|^2+\mu_\mathbf{u}\|\nabla_h\mathbf{u}\|^2
+\nu_{\mathbf{u}}\|\partial_p\mathbf{u}\|_w^2
\leq-\int_\mathcal M\frac{RT}{p}\omega d\mathcal M.
\end{equation}

Multiplying equation (\ref{AT}) by $c_p$, integrating the resultant over $\mathcal M$, and recalling the nonnegativity of $T$, it follows from integration by parts that
\begin{equation}
  c_p\frac{d}{dt}\|T\|_{L^1}=\int_\mathcal M
  \left(\mathcal D^TT+\frac{ RT}{p}\omega+L(S_{cd}-S_{ev,\varepsilon})\right)d\mathcal M. \label{egyT.1}
\end{equation}
By the boundary conditions (\ref{bound.0})--(\ref{bound.ll}) for $T$,
it follows from integration by parts that
\begin{eqnarray}
\int_\mathcal M\mathcal D^TTd\mathcal M
&=&\int_\mathcal M\left(
\mu_T\Delta_hT+\nu_T\partial_p\left(\frac{gp}{R_d\bar T}\partial_pT\right)\right)d\mathcal M\nonumber\\
&=&\mu_T\int_{\Gamma_\ell}\partial_{\mathbf{n}}Td\Gamma_\ell+\left.\nu_T \int_{\mathcal M'}\frac{gp}{R_d\bar T}\partial_pTd\mathcal M'\right|_{p=p_1}^{p_0}\nonumber\\
&=&\mu_T\int_{\Gamma_\ell}\alpha_{\ell T}(T_{b\ell}-T)d\Gamma_\ell
+\frac{\nu_Tgp_0\alpha_{0T}}{R_d\bar T(p_0)}\left.\int_{\mathcal M'}(T_{b0}-T)d\mathcal M'\right|_{p_0}\nonumber\\
&\leq&\mu_T\int_{\Gamma_\ell}\alpha_{\ell T}T_{b\ell}d\Gamma_\ell
+\frac{\nu_Tgp_0\alpha_{0T}}{R_d\bar T(p_0)}\int_{\mathcal M'}T_{b0}d\mathcal M'\leq C, \label{egyT.2}
\end{eqnarray}
where we note that $C$ depends on the boundary data. Recalling (\ref{bddsource}) we have moreover
\begin{equation}
\int_\mathcal M L(S_{cd}-S_{ev,\varepsilon})d\mathcal M\leq C \label{egyT.3}
\end{equation}
Thanks to (\ref{egyT.2}) and (\ref{egyT.3}), it follows from (\ref{egyT.1}) that
\begin{equation}\label{egyT}
  c_p\frac{d}{dt}\|T\|_{L^1}\leq
\int_\mathcal M\frac{RT}{p}\omega d\mathcal M+C,
\end{equation}
with the constant $C$ depending again on the boundary data.
Adding (\ref{egyu}) and (\ref{egyT}) yields
\begin{equation*}
\frac{d}{dt}\left(\frac12\|\mathbf{u}\|^2+c_p\|T\|_{L^1}\right)
+\mu_\mathbf{u}\|\nabla_h\mathbf{u}\|^2
+\nu_{\mathbf{u}}\|\partial_p\mathbf{u}\|_w^2
\leq C,
\end{equation*}
from which by integration the conclusion of the proposition follows.
\end{proof}

Thanks to the basic energy estimate, we can now derive the $L^\infty(L^2)\cap L^2(H^1)$ estimate on $T$ contained in the following proposition.

\begin{prop}
  \label{L2T}
Let Assumption \ref{ass} hold, then there exists a function $K_2(t)$, which depends on the initial and boundary data, and is continuous for all $t\geq 0$, such that
\begin{equation*}
  \sup_{0\leq t\leq\mathcal T}\|T\|^2+\int_0^{\mathcal T}(\|\nabla_hT\|^2+\|\partial_pT\|_w^2)dt\leq K_2(\mathcal T),
\end{equation*}
for any $\mathcal T\in[0,\mathcal T_{\text{max}})$.
\end{prop}

\begin{proof}
Performing integration by parts and using the boundary conditions for $T$ we obtain
\begin{eqnarray}
  -\int_{\mathcal M}T \mathcal D^T Td\mathcal M&=&-\int_{\mathcal M}
  \left[\mu_T\Delta_hT+\nu_T\partial_p\left(\left(\frac{gp}{R_d\bar
T}\right)^2\partial_pT\right)\right]T d\mathcal M\nonumber\\
  &=&-\mu_T\int_{\Gamma_\ell}\partial_\mathbf{n}TT
d\Gamma_\ell-\nu_T\int_{\mathcal M'}\left(\frac{gp}{R_d\bar
T}\right)^2\partial_pTT d\mathcal
M'\bigg|_{p=p_1}^{p_0}\nonumber\\
  &&+\int_{\mathcal M}\left[\mu_T\nabla_hT\cdot\nabla_hT +\nu_T
  \left(\frac{gp}{R_d\bar
T}\right)^2\partial_pT\partial_pT \right] d\mathcal M\nonumber\\
  &=&\mu_T\|\nabla_hT \|^2+\nu_T\|\partial_pT \|_w^2
  -\mu_T\int_{\Gamma_\ell}\alpha_{\ell T}(T_{b\ell}-T)T
d\Gamma_\ell
  \nonumber\\
  &&-\nu_T\alpha_{0T}\left(\frac{gp_0}{R_d\bar
T(p_0)}\right)^2\left.\int_{\mathcal M'}
  (T_{b0}-T)T d\mathcal M'\right|_{p_0}\nonumber\\
&\geq&\mu_T\|\nabla_hT \|^2+\nu_T\|\partial_pT \|_w^2
-\mu_T\int_{\Gamma_\ell}\alpha_{\ell T}T_{b\ell}T
d\Gamma_\ell\nonumber\\
&&-\nu_T\alpha_{0T}\left.\left(\frac{gp_0}{R_d\bar
T(p_0)}\right)^2\int_{\mathcal M'}
  T_{b0}T d\mathcal M'\right|_{p_0}.\nonumber
\end{eqnarray}
By the trace inequality, the boundary integrals in the above inequality can be bounded as
\begin{eqnarray*}
 &&\mu_T\int_{\Gamma_\ell}\alpha_{\ell T}T_{b\ell}T
d\Gamma_\ell+\nu_T\alpha_{0T}\left(\frac{gp_0}{R_d\bar
T(p_0)}\right)^2\left.\int_{\mathcal M'}
  T_{b0}T d\mathcal M'\right|_{p_0}\\
&\leq& C\|T\|_{L^2(\partial\mathcal M)}\leq C\|T\|_{H^1(\mathcal M)}
\leq \frac{\mu_T}{4}\|\nabla_hT \|^2+\frac{\nu_T}{4}\|\partial_pT \|_w^2
+C(1+\|T\|^2).
\end{eqnarray*}
Therefore, we have
\begin{equation}
  -\int_{\mathcal M}T \mathcal D^T Td\mathcal M\geq
  \frac{3\mu_T}{4}\|\corr{\nabla_h}T\|^2+\frac{3\nu_T}{4}
\|\partial_pT\|_w^2-C(1+\|T\|^2).\label{LI-2}
\end{equation}

Multiplying (\ref{AT}) with $T$ and using (\ref{LI-2}) we get
\begin{eqnarray*}
&&\frac{1}{2}\frac{d}{dt}\|T\|^2
+ \frac{3\mu_T}{4}\|\nabla_h T\|^2 +\frac{3\nu_T}{4}\|\pa_p T\|_w^2
\nonumber\\
&\leq& \int_\M \frac{R}{c_p}\frac{\omega}{p}T^2 d\M
+ \int_\M \frac{L}{c_p}(S_{cd}-S_{ev,\varepsilon})  T d\M,
\end{eqnarray*}
from which, recalling (\ref{bddsource}), we obtain
\begin{equation}
\frac{d}{dt}\|T \|^2
+ \frac{3\mu_T}{2}\|\nabla_h T \|^2 +\frac{3\nu_T}{2}\|\pa_p T \|_w^2
\leq C\int_{\mathcal M}(|\omega|T^2+1+T^2)d\mathcal M.
\label{LI-4}
\end{equation}
By Lemma \ref{lemlad} from the Appendix and Young's inequality, we deduce
\begin{align}
C\int_{\mathcal M}&|\omega|T^2d\mathcal M=\int_\mathcal M\left|\int_p^{p_0}\nabla_h\cdot\mathbf{u}
dp'\right|T^2d\mathcal M\nonumber\\
\leq&\  C\int_{\mathcal M'} \int_{p_1}^{p_0}|\nabla_h\mathbf{u}|dp\int_{p_1}^{p_0}T^2dp d\mathcal M' \leq C\|\nabla_h\mathbf{u}\|\|T\|(\|T\|+\|\nabla_hT\|)\nonumber\\
\leq&\ \frac{\mu_T}{2}\|\nabla_hT\|^2+C(1+\|\nabla_h\mathbf{u}\|^2) \|T\|^2,\nonumber
\end{align}
which, substituted into (\ref{LI-4}), leads to
$$
\frac{d}{dt}\|T \|^2
+  \mu_T \|\nabla_h T \|^2 + \nu_T \|\pa_p T \|_w^2
\leq C(1+\|\nabla_h\mathbf{u}\|_2^2)(1+\|T\|^2),
$$
from which, by the Gronwall inequality and Proposition \ref{basic}, the conclusion follows.
\end{proof}

We have the following $L^\infty(L^2)$  estimate on the vertical derivative of the velocity, which was first established by Cao and Titi in \cite{CT} for the primitive equations without coupling to the
moisture equations.

\begin{prop}
\label{est.u.pu}
Let Assumption \ref{ass} hold, then there exists a function $K_3(t)$, which depends on the initial and boundary data, and is continuous for all $t\geq 0$, such that
\begin{equation*}
 \sup_{0\leq t\leq\mathcal T}
(\|\mathbf{u}\|_{L^6}^6+\|\pa_p {\mathbf{u}}\|^2) + \int_0^{\mathcal T}(\|\pa_p\nabla_h\mathbf{u}\|^2 + \|\pa_p^2\mathbf{u}\|_w^2)dt
\leq K_3(\mathcal T),
\end{equation*}
for any $\mathcal T\in[0,\mathcal T_{\text{max}})$.
\end{prop}

\begin{proof}
The proof is adapted from that in \cite{CT}. Decompose the velocity into the barotropic and baroclinic modes $\bar{\mathbf{u}}$ and $\widetilde{\mathbf{u}}$ as follows
$$
\mathbf{u}=\bar{\mathbf{u}}+\widetilde{\mathbf{u}},
$$
where
$\bar{f}:=\frac{1}{p_0-p_1}\int_{p_1}^{p_0}fdp$.
For the $L^6$-norm of $\widetilde{\mathbf u}$, following the arguments in \cite{CT}, with tiny modifications on dealing with the integral involving $\Phi$, we have (see the inequality above (66) in page 257 of \cite{CT})
\begin{eqnarray*}
&&\frac{d}{dt}\|\widetilde{\mathbf{u}}\|_{L^6}^6 +  \int_\M \big(\mu_\mathbf{u}|\nabla_h  \widetilde{\mathbf{u}}|^2|\widetilde{\mathbf{u}}|^4  + \nu_\mathbf{u}  |\pa_p  \widetilde{\mathbf{u}}|^2|\widetilde{\mathbf{u}}|^4\big)d\M\nonumber\\
& \leq& C\big(\|\bar{\mathbf{u}}\|^2\|\nabla\bar{\mathbf{u}}\|^2 +\|\nabla\widetilde{\mathbf{u}}\|^2  +\|\widetilde{\mathbf{u}}\|^2 \big)\|\widetilde{\mathbf{u}}\|_6^6 + C\|\bar T\|^2\|\nabla\bar T\|^2.
\end{eqnarray*}
Noticing $\|\bar f\| +  \|\tilde f\|\leq
C\|f\|$, $\|\nabla \bar f\|+\|\nabla\tilde
f\|\leq C\|\nabla f\|$,  it follows from Proposition \ref{basic},  Proposition \ref{L2T} and the Gronwall inequality that
\begin{equation}
\sup_{0\leq t\leq\mathcal T}
\|\widetilde{\mathbf{u}}\|_{L^6}^6 +  \int_0^{\mathcal T} \int_\M \big( |\nabla_h  \widetilde{\mathbf{u}}|^2|\widetilde{\mathbf{u}}|^4   +  |\pa_p  \widetilde{\mathbf{u}}|^2|\widetilde{\mathbf{u}}|^4  \big) d\M dt\leq K_3'(\mathcal T),\label{est.u.L6}
\end{equation}
for a continuous function $K_3'$ on $[0,\infty)$, which depends on the initial and boundary data.
For the barotropic mode $\bar{\mathbf{u}}$ it holds that (see the first inequality in page 259 of \cite{CT})
\begin{equation*}
\frac{d}{dt}\|\nabla_h\bar{\mathbf{u}}\|^2 + \mu_\mathbf{u}\|\Delta_h\bar{\mathbf{u}}\|^2 \leq  C \|\bar{\mathbf{u}}\|^2 \|\nabla_h\bar{\mathbf{u}}\|^4+C\Big( \|\nabla_h \bar{\mathbf{u}}\|^2  + \int_\M |\nabla_h  \widetilde{\mathbf{u}}|^2|\widetilde{\mathbf{u}}|^4 d\M + \|\bar{\mathbf{u}}\|^2 \Big),
\end{equation*}
from which, by Proposition \ref{basic} and (\ref{est.u.L6}), it follows from the Gronwall inequality that
\begin{equation}\label{est.nabla.baru}
\sup_{0\leq t\leq\mathcal T}
\|\nabla_h\bar{\mathbf{u}}\|^2 + \int_0^{\mathcal T}\|\Delta_h\bar{\mathbf{u}}\|^2dt\leq K_3''(\mathcal T),
\end{equation}
for a continuous function $K_3''$ on $[0,\infty)$.
Note that by the Sobolev embedding inequality and Proposition \ref{basic}
one can easily obtain from (\ref{est.u.L6}) and (\ref{est.nabla.baru}) that
\begin{equation*}
  \sup_{0\leq t\leq\mathcal T}\|\mathbf{u}\|_{L^6}\leq K_3'''(\mathcal T),
  \label{est.L6}
\end{equation*}
for a continuous function $K_3'''$ on $[0,\infty)$ depending on the initial and boundary data.
For the vertical gradient we have the following bound (see the inequality above (75) in page 260 of \cite{CT})
$$
\frac{d}{dt}\|\pa_p {\mathbf{u}}\|^2 + \mu_\mathbf{u}\|\pa_p\nabla_h\mathbf{u}\|^2 + \nu_\mathbf{u}\|\pa_p^2\mathbf{u}\|_w^2
\leq  C (\|\nabla_h\bar{\mathbf{u}}\|^4 +  \|\widetilde{\mathbf{u}}\|_{L^6}^4)\|\pa_p \mathbf{u}\|^2 + C\|T\|^2  \,,
$$
from which, by (\ref{est.u.L6})--(\ref{est.nabla.baru}), Proposition \ref{L2T} and the Gronwall inequality it follows that
\begin{equation*}
 \sup_{0\leq t\leq\mathcal T}
\|\pa_p {\mathbf{u}}\|^2 + \int_0^{\mathcal T}(\|\pa_p\nabla_h\mathbf{u}\|^2 + \|\pa_p^2\mathbf{u}\|_w^2)dt
\leq K_3''''(\mathcal T),
\end{equation*}
for a continuous function $K_3''''$ on $[0,\infty)$ depending on the initial and boundary data, which  completes the proof.
\end{proof}

We are now ready to establish the estimates for the vertical derivative of the moisture components.

\begin{prop}
\label{prop.est.pq}
Let Assumption \ref{ass} hold, then there exists a function $K_4(t)$, which depends on the initial and boundary data, and is continuous for all $t\geq 0$, such that
  \begin{eqnarray*}
  \sup_{0\leq t\leq\mathcal T}\|\partial_pq_j\|^2
+\int_0^\mathcal T\|\nabla\partial_pq_j\|^2dt\leq K_4(\mathcal T),\quad j\in\{v,c,r\},
\end{eqnarray*}
for any $\mathcal T\in[0,\mathcal T_{\text{max}})$.
\end{prop}
\begin{proof}
We only prove the estimate for $q_r$, since the derivation for the other moisture quantities follows the same steps.
Multiplying equation (\ref{Aqr}) by $-\partial_p^2q_r$ and integrating over $\mathcal M$ yield
\begin{eqnarray}
&&-\int_\mathcal M\partial_tq_r\partial_p^2q_rd\mathcal M+\int_\mathcal M
\mathcal D^{q_r}q_r\partial_p^2q_rd\mathcal M\nonumber\\
&=&\int_\mathcal M(\mathbf{u}\cdot\nabla_hq_r+\omega\partial_pq_r)\partial_p^2q_rd\mathcal M+V\int_\mathcal M\partial_p\left(\frac{pq_r}{R_d\bar T}\right)\partial_p^2q_rd\mathcal M\nonumber\\
&&+\int_\mathcal M(S_{ev,\varepsilon}-S_{ac}-S_{cr})\partial_p^2q_rd\mathcal M\,,
\label{est.pqr}
\end{eqnarray}
Recalling (\ref{bddsource}), it follows from the H\"older and Young inequalities that
\begin{eqnarray}
  V\int_\mathcal M\partial_p\left(\frac{pq_r}{R_d\bar T}\right)\partial_p^2q_rd\mathcal M+\int_\mathcal M(S_{ev,\varepsilon}-S_{ac}-S_{cr})\partial_p^2q_rd\mathcal M\nonumber\\
\leq\ \frac{\nu_{q_r}}{16}\|\partial_p^2q_r\|^2+C(\|\partial_pq_r\|^2+1).
\label{soupq}
\end{eqnarray}

For the fist term on the left-hand side of (\ref{est.pqr}), using the boundary conditions (\ref{bound.0}) and (\ref{bound.1}), it follows from integration by parts that
\begin{eqnarray}
&&-\int_\mathcal M\partial_tq_r\partial_p^2q_rd\mathcal M\nonumber\\
&=&-\left.\int_{\mathcal M'}\partial_tq_r\partial_pq_rd\mathcal M'
\right|_{p_0}+\int_\mathcal M\partial_t\partial_pq_r\partial_pq_r d\mathcal M\nonumber\\
&=&\frac12\frac{d}{dt}\|\partial_pq_r\|^2-\alpha_{0r}\left.\int_{\mathcal M'}\partial_tq_r(q_{b0r}-q_r)d\mathcal M'\right|_{p_0}\nonumber\\
&=&\frac{d}{dt}\left(\frac{\|\partial_pq_r\|^2}{2}+\alpha_{0r}
\left.\int_{\mathcal M'}\left(\frac{q_r^2}{2}-q_rq_{b0r}\right) d\mathcal M'\right|_{p_0}\right)+\alpha_{0r}\left.\int_{\mathcal M'} q_r\partial_tq_{b0r}d\mathcal M'\right|_{p_0}\nonumber\\
&\geq&\frac{d}{dt}\left(\frac{\|\partial_pq_r\|^2}{2}+\alpha_{0r}
\left.\int_{\mathcal M'}\left(\frac{q_r^2}{2}-q_rq_{b0r}\right) d\mathcal M'\right|_{p_0}\right)-C(\|\partial_pq_r\|+1),
\label{tpqr}
\end{eqnarray}
where in the last step we have used Lemma \ref{lem} (see Appendix), Proposition \ref{bddq} and the H\"older inequality to estimate
\begin{eqnarray*}
  \left|\left.\int_{\mathcal M'} q_r\partial_tq_{b0r}d\mathcal M'\right|_{p_0}\right|
\leq C(\|q_r\|_{L^1(\mathcal M)}+\|\partial_pq_r\|_{L^1(\mathcal M)})
\leq C(1+\|\partial_pq_r\|).
\end{eqnarray*}
We next bound  the integrals involving the diffusion terms. It should be noticed that the following derivations are formal, since the regularity of $q_r$ does not guarantee the validity of the integrals on the boundary $\Gamma_0=\M'\times\{p_0\}$ directly. However, due to the $C^2(\overline\M\times[0,T])$ functions, enjoying the boundary conditions (\ref{bound.0})--(\ref{bound.ll}) for $q_r$, being dense in the space $\{f|f\in L^2(0,\mathcal T; H^2(\M)), \partial_tf\in L^2(0,\mathcal T; L^2(\M))\}$, which $q_r$ belongs to, one can rigorously justify the desired estimates (\ref{df.npqr}) below for $q_r$ in the standard way ( i.e., choosing a sequence $\{q_{rn}\}_{n=1}^\infty$ in $C^2(\overline\M\times[0,\mathcal T])$ satisfying the corresponding boundary conditions (\ref{bound.0})--(\ref{bound.ll}), proving that (\ref{df.npqr}) holds for each $q_{rn}$, and passing the limit $n\rightarrow\infty$).

Using the boundary conditions (\ref{bound.1}), it follows from integration by parts that
\begin{eqnarray}
&&\int_\mathcal M\mathcal D^{q_r}q_r\partial_p^2q_rd\mathcal M\nonumber\\
&=&\int_\mathcal M\left[\mu_{q_r}\Delta_hq_r+\nu_{q_r}\partial_p\left(\left(
\frac{gp}{R_d\bar T}\right)^2\partial_pq_r\right)\right]\partial_p^2 q_rd\mathcal M\nonumber\\
&=&\mu_{q_r}\left.\int_{\mathcal M'}\Delta_hq_r\partial_pq_rd\mathcal M'\right|_{p_0}-\mu_{q_r}\int_\mathcal M\partial_p\Delta_hq_r\partial_p q_rd\mathcal M\nonumber\\
&&+\nu_{q_r}\|\partial_p^2q_r\|_w^2+\nu_{q_r}\int_\mathcal M\partial_p \left(\frac{gp}{R_d\bar T}\right)^2\partial_pq_r\partial_p^2q_rd\mathcal M\nonumber\\
&=&\mu_{q_r}\left(\left.\int_{\mathcal M'}\Delta_hq_r\partial_pq_rd\mathcal M'
\right|_{p_0}- \int_{\Gamma_\ell}\partial_p\partial_{\textbf{n}}q_r
\partial_pq_rd\Gamma_\ell + \|\nabla_h\partial_pq_r\|^2\right)\nonumber\\
&&+\nu_{q_r}\|\partial_p^2q_r\|_w^2+\nu_{q_r}\int_\mathcal M\partial_p
\left(\frac{gp}{R_d\bar T}\right)^2\partial_pq_r\partial_p^2q_rd\mathcal M\nonumber\\
&\geq&\mu_{q_r}\left(\left.\int_{\mathcal M'}\Delta_hq_r\partial_pq_rd\mathcal M'
\right|_{p_0}- \int_{\Gamma_\ell}\partial_p\partial_{\textbf{n}}q_r
\partial_pq_rd\Gamma_\ell \right)\nonumber\\
&&+\mu_{q_r}\|\nabla_h\partial_pq_r\|^2+\frac{3\nu_{q_r}}{4}
\|\partial_p^2q_r\|_w^2-C\|\partial_pq_r\|^2. \label{df.npqr0}
\end{eqnarray}
Integrating by parts and using the boundary conditions (\ref{bound.0}) and (\ref{bound.ll}),
we deduce
\begin{eqnarray}
&&\left.\int_{\mathcal M'}\Delta_hq_r\partial_pq_rd\mathcal M'
\right|_{p_0}=\left.\int_{\partial\mathcal M'}\partial_\textbf{n}q_r\partial_p
q_rd\partial\mathcal M'\right|_{p_0}-\left.\int_{\mathcal M'}\nabla_hq_r\cdot\nabla_h\partial_pq_rd\mathcal M'\right|_{p_0}\nonumber\\
&=&\left.\alpha_{\ell r}\alpha_{0r}\int_{\partial\mathcal M'}
(q_{b\ell r}-q_r)(q_{b0r}-q_r)d\partial\mathcal M'\right|_{p_0}
+\alpha_{0r}\left.\int_{\mathcal M'}\nabla_hq_r\cdot\nabla_h(q_r -q_{b0r})d\mathcal M'\right|_{p_0}
\nonumber\\
&\geq&-\frac{\alpha_{\ell r}\alpha_{0r}}{4}
\left.\int_{\partial\mathcal M'}
(q_{b\ell r}-q_{b0r})^2d\partial\mathcal M'\right|_{p_0}-\frac{\alpha_{0r}}{4}\left.\int_{\mathcal M'}|\nabla_hq_{b0r}|^2d\mathcal M'\right|_{p_0}\geq-C,\label{df.npqr1}
\end{eqnarray}
where we have used
\begin{eqnarray}
\label{est.qbl}&(q_{b\ell r}-q_r)(q_{b0r}-q_r)=\left(q_r-\frac{q_{b\ell r}+q_{b0r}}{2}
\right)^2-\frac{(q_{b\ell r}-q_{b0r})^2}{4}\geq -\frac{(q_{b\ell r}-q_{b0r})^2}{4}, \\
\label{est.qblgrad}&\nabla_hq_r\cdot\nabla_h(q_r -q_{b0r})=\left|\nabla_hq_r-\frac{\nabla_h q_{b0r}}{2}\right|^2-\frac{|\nabla_hq_{b0r}|^2}{4} \geq-\frac{|\nabla_hq_{b0r}|^2}{4}.
\end{eqnarray}
Using the boundary condition (\ref{bound.ll}), one has
\begin{eqnarray}
-\int_{\Gamma_\ell}\partial_p\partial_{\textbf{n}}q_r\partial_pq_rd
\Gamma_\ell&=&\alpha_{\ell r}\int_{\Gamma_\ell}\partial_p(q_r-q_{b\ell r})\partial_pq_rd\Gamma_\ell\nonumber\\
&=&\alpha_{\ell r}\int_{\Gamma_\ell}\left[\left(\partial_pq_r-\frac{\partial_pq_{b\ell r}}{2}\right)^2-\frac{|\partial_pq_{b\ell r}|^2}{4}\right] d\Gamma_\ell
\nonumber\\
&\geq&-\frac{\alpha_{\ell r}}{4}\int_{\Gamma_\ell}|\partial_pq_{b\ell r}|^2 d\Gamma_\ell\geq-C.\label{df.npqr2}
\end{eqnarray}
Substituting (\ref{df.npqr1}) and (\ref{df.npqr2}) into (\ref{df.npqr0})
yields
\begin{equation}
  \int_\mathcal M\mathcal D^{q_r}q_r\partial_p^2q_rd\mathcal M
\geq \frac34\left(\mu_{q_r}\|\nabla_h\partial_pq_r\|^2+\nu_{q_r}
\|\partial_p^2q_r\|_w^2\right)-C(\|\partial_pq_r\|^2+1).
\label{df.npqr}
\end{equation}

For the term involving the convection, it follows from integration by parts and the boundary condition (\ref{bound.1}) that
\begin{eqnarray}
  &&\int_\mathcal M(\mathbf{u}\cdot\nabla_hq_r+\omega\partial_p q_r)
\partial_p^2q_rd\mathcal M\nonumber\\
&=&\left.\int_{\mathcal M'}(\mathbf{u}\cdot\nabla_hq_r+\omega\partial_pq_r)\partial_pq_rd\mathcal M'\right|_{p_0}-\int_\mathcal M(\partial_p\mathbf{u}\cdot\nabla_hq_r-\nabla_h\cdot\mathbf{u}\partial_pq_r)
\partial_pq_r d\mathcal M\nonumber\\
&&-\int_\mathcal M(\mathbf{u}\cdot\nabla_h\partial_pq_r+\omega\partial_p^2q_r)\partial_pq_r d\mathcal M\nonumber\\
&=&\alpha_{0r}\left.\int_{\mathcal M'}\mathbf{u}\cdot\nabla_hq_r(q_{b0r}-q_r)d\mathcal M'\right|_{p_0}            -\int_\mathcal M(\partial_p\mathbf{u}\cdot\nabla_hq_r-\nabla_h\cdot\mathbf{u}\partial_pq_r)
\partial_pq_r d\mathcal M.\label{covpq.0}
\end{eqnarray}
By Proposition \ref{bddq}, Lemma \ref{lem} (see Appendix), and Young's inequality, we obtain
\begin{eqnarray}
&&\alpha_{0r}\left|\left.\int_{\mathcal M'}\mathbf{u}\cdot\nabla_hq_r(q_{b0r}-q_r)d\mathcal M'\right|_{p_0}\right|\nonumber\\
&\leq& C\left.\int_{\mathcal M'}|\mathbf{u}\cdot\nabla_hq_r|d\mathcal M'\right|_{p_0}
\leq C(\|\mathbf{u}\cdot\nabla_hq_r\|_{L^1(\mathcal M)}+\|\partial_p
(\mathbf{u}\cdot\nabla_hq_r)\|_{L^1(\mathcal M)})\nonumber\\
&\leq& C(\|\mathbf{u}\cdot\nabla_hq_r\|_{L^1(\mathcal M)}+\|\partial_p
\mathbf{u}\cdot\nabla_hq_r\|_{L^1(\mathcal M)}+\|\mathbf{u}\cdot\nabla_h
\partial_pq_r\|_{L^1(\mathcal M)})\nonumber\\
&\leq& \frac{\mu_{q_r}}{12}\|\nabla_h\partial_pq_r\|_w^2+ C(\|\mathbf{u}\|^2+\|\nabla_hq_r\|^2+\|\partial_p\mathbf{u}\|^2)
\nonumber\\
&\leq& \frac{\mu_{q_r}}{12}\|\nabla_h\partial_pq_r\|_w^2+ C(\|\nabla_hq_r\|^2+1),
\label{covpq.1}
\end{eqnarray}
where in the last step we have used Proposition \ref{basic} and Proposition \ref{est.u.pu}.
Note that the boundary condition (\ref{bound.ll}) for $\mathbf{u}$
implies $\partial_p\mathbf{u}\cdot\mathbf{n}=0$ on $\Gamma_\ell$. Thanks
to this, by Proposition \ref{bddq} and Proposition \ref{est.u.pu},
it follows from integration by parts and the Young inequality that
\begin{eqnarray}
-\int_\mathcal M\partial_p\mathbf{u}\cdot\nabla_hq_r\partial_pq_rd\mathcal M
&=&\int_\mathcal M(\partial_p\nabla_h\cdot\mathbf{u}\partial_pq_r +\partial_p\mathbf{u}\cdot\nabla_h\partial_pq_r)q_rd\mathcal M
\nonumber\\
&\leq&\frac{\mu_{q_r}}{12}\|\nabla_h\partial_pq_r\|^2+C(\|\partial_p\nabla_h
\mathbf{u}\|^2+\|\partial_pq_r\|^2+1). \label{covpq.2}
\end{eqnarray}
Noticing  $\int_{p_1}^{p_0}\nabla_h\cdot\mathbf{u}dp=0$, we have
\begin{eqnarray}
|\nabla_h\cdot\mathbf{u}(x,y,p)|&=&\left|\frac{1}{p_0-p_1}\int_{p_0}^{p_1}
\nabla_h\cdot\mathbf udq
+\frac{1}{p_0-p_1}\int_{p_1}^{p_0}\int_q^p\nabla_h\cdot\partial_p\mathbf{u}
d\xi dq\right|\nonumber\\
&=&\left|\frac{1}{p_0-p_1}\int_{p_1}^{p_0}\int_q^p\nabla_h\cdot\partial_p\mathbf{u}
d\xi dq\right|\leq\int_{p_1}^{p_0}|\nabla_h\partial_p\mathbf{u}|dq,
\end{eqnarray}
for any $p\in[p_1,p_0]$. Based upon this fact, it follows from Lemma \ref{lemlad} (see Appendix) that
\begin{eqnarray}
  \left|\int_\mathcal M\nabla_h\cdot\mathbf{u}(\partial_pq_r)^2d\mathcal M\right|
&\leq&\int_{\mathcal M'}\int_{p_1}^{p_0}|\nabla_h\partial_p\mathbf{u}|dp
\int_{p_1}^{p_0}|\partial_pq_r|^2dpd\mathcal M'\nonumber\\
&\leq& C\|\nabla_h\partial_p\mathbf{u}\|\|\partial_pq_r\|(\|\partial_pq_r\|+
\|\nabla_h\partial_pq_r\|)\nonumber\\
&\leq&\frac{\mu_{q_r}}{12}\|\nabla_h\partial_pq_r\|^2+C(1+\|\nabla_h\partial_p \mathbf{u}\|^2)\|\partial_pq_r\|^2.\label{covpq.3}
\end{eqnarray}
Substituting (\ref{covpq.1}), (\ref{covpq.2}), and (\ref{covpq.3}) into (\ref{covpq.0}), one obtains the estimate
\begin{eqnarray}
  &&\int_\mathcal M(\mathbf{u}\cdot\nabla_hq_r+\omega\partial_pq_r)
\partial_p^2q_rd\mathcal M\nonumber\\
&\leq&\frac{\mu_{q_r}}{4}\|\nabla_h\partial_pq_r\|^2+C(
\|\nabla_hq_r\|^2+\|\nabla_h\partial_p\mathbf{u}\|^2+1)
(1+\|\partial_pq_r\|^2).\label{covpq}
\end{eqnarray}

Substituting (\ref{soupq}), (\ref{tpqr}), (\ref{df.npqr}), and (\ref{covpq}) into (\ref{est.pqr}) leads to
\begin{eqnarray}
  \frac{d}{dt}\left(\frac{\|\partial_pq_r\|^2}{2}+\alpha_{0r}\int_{\mathcal M'}\left.\left(\frac{q_r^2}{2}-q_rq_{b0r}\right)d\mathcal M'\right|_{p_0}\right)+\frac12(\mu_{q_r}\|\nabla_h\partial_pq_r\|^2 +\nu_{q_r}\|\partial_p^2q_r\|_w^2)\nonumber\\
\leq C(\|\nabla_h\partial_p\mathbf{u}\|^2+\|\nabla_hq_r\|^2+1)
(1+\|\partial_pq_r\|^2).\nonumber
\end{eqnarray}
from which, noticing that
\begin{eqnarray*}
\left.\int_{\mathcal M'}(q_r^2-2q_rq_{b0r})d\mathcal M'\right|_{p_0}
&=&\left.\int_{\mathcal M'}\left[(q_r-q_{b0r})^2-q_{b0r}^2\right]d\mathcal M'\right|_{p_0}\\
&\geq&-\left.\int_{\mathcal M'}q_{b0r}^2 d\mathcal M'\right|_{p_0}\geq-C,
\end{eqnarray*}
the conclusion follows, by Proposition \ref{H1q}, Proposition \ref{est.u.pu}, and the Gronwall inequality.
\end{proof}

Based on the estimates on the vertical derivatives of the velocity and
moisture components, we can now bound the corresponding horizontal derivatives.

\begin{prop}
\label{prop.est.hu}
Let Assumption \ref{ass} hold, then there exists a function $K_5(t)$, which depends on the initial and boundary data, and is continuous for all $t\geq 0$, such that
\begin{eqnarray*}
\sup_{0\leq t\leq\mathcal T}\|\nabla_h\mathbf{u}\|^2
+\int_0^\mathcal T\left(
\|\Delta_h\mathbf{u}\|^2 + \|\pa_p\nabla_h
\mathbf{u}\|^2\right)dt\leq K_5(\mathcal T),
\end{eqnarray*}
and
\begin{eqnarray*}
\sup_{0\leq t\leq\mathcal T}
\|\nabla_hq_j\|^2+\int_0^\mathcal T\left(\|\Delta_hq_j\|^2
+\|\nabla_h\partial_pq_j\|_w^2\right)dt\leq K_5(\mathcal T),\quad j\in\{v,c,r\},
\end{eqnarray*}
for any $\mathcal T\in[0,\mathcal T_{\text{max}})$.
\end{prop}

\begin{proof}
Following the arguments in Section 3.3.3 of \cite{CT}, we have for the estimate of the horizontal gradient,
\begin{eqnarray*}
&&\frac{d}{dt}\|\nabla_h\mathbf{u}\|^2 + \mu_\mathbf{u}
\|\Delta_h\mathbf{u}\|^2 + \nu_\mathbf{u}\|\pa_p\nabla_h
\mathbf{u}\|^2\nonumber\\
&& \quad \leq C (\|\mathbf{u}\|_{L^6}^4 +  \|\nabla_h \mathbf{u}
\|^2 \|\pa_p \mathbf{u}\|^2)\|\nabla_h \mathbf{u}\|^2
+C\int_\mathcal M \left(\int_{p_1}^p|\nabla_h T
|d\sigma\right)|\Delta_h\mathbf{u}|d\mathcal M,
\end{eqnarray*}
from which, by Proposition \ref{est.u.pu}  and Young's inequality, one obtains
\begin{eqnarray}
&&\frac{d}{dt}\|\nabla_h\mathbf{u}\|^2 + \mu_\mathbf{u}
\|\Delta_h\mathbf{u}\|^2 + \nu_\mathbf{u}\|\pa_p\nabla_h
\mathbf{u}\|^2\nonumber\\
&&\quad \leq  C (1 +  \|\nabla_h \mathbf{u}
\|^2 )\|\nabla_h \mathbf{u}\|^2
+C\|\nabla_hT\|\|\Delta_h\mathbf{u}\|\nonumber\\
&&\quad \leq  C (1 +  \|\nabla_h \mathbf{u}
\|^2 )\|\nabla_h \mathbf{u}\|^2
+C\|\nabla_hT\|^2 +\frac{\mu_{\mathbf{u}}}{4}\|\Delta_h\mathbf{u}\|^2,\label{est.u.hu.0}
\end{eqnarray}
leading to
\begin{eqnarray}
\frac{d}{dt}\|\nabla_h\mathbf{u}\|^2 + \frac{3\mu_\mathbf{u}}{4}
\|\Delta_h\mathbf{u}\|^2 + \nu_\mathbf{u}\|\pa_p\nabla_h
\mathbf{u}\|^2\leq  C (1+  \|\nabla_h \mathbf{u}
\|^2 )\|\nabla_h \mathbf{u}\|^2+ C \|\nabla_hT\|^2.\label{est.u.hu}
\end{eqnarray}
Next, we estimate the horizontal gradient of $q_r$.
Multiplying equation (\ref{Aqr}) by $-\Delta_hq_r$ and integrating over
$\mathcal M$ yield
\begin{eqnarray}
&&\int_\mathcal M\left(-\partial_tq_r+
\label{est.hqr.0} D^{q_r}q_r\right)\Delta_hq_rd\mathcal M\\
&&=\int_\mathcal M \left[V\partial_p\left(\frac{pq_r}{R_d\bar T}\right)+
S_{ac}+S_{cr}-S_{ev,\varepsilon}\right]\Delta_hq_rd\mathcal M+\int_\mathcal M(\mathbf{u}\cdot\nabla_hq_r+\omega\partial_pq_r)\Delta_h
q_rd\mathcal M.
\nonumber
\end{eqnarray}
Recalling (\ref{bddsource}), it follows from the Young
inequality, Proposition \ref{bddq} and Proposition \ref{prop.est.pq} that
\begin{eqnarray}
&&\int_\mathcal M \left[V\partial_p\left(\frac{pq_r}{R_d\bar T}\right)
+S_{ac}+S_{cr}-S_{ev,\varepsilon}\right]\Delta_hq_rd\mathcal M\nonumber\\
&&\leq\frac{\mu_{q_r}}{4}\|\Delta_hq_r\|^2+C(\|q_r\|^2+\|\partial_p
q_r\|^2) \leq\frac{\mu_{q_r}}{8}\|\Delta_hq_r\|^2+C.
\label{est.hqr.1}
\end{eqnarray}
Integrating by parts and using the boundary condition (\ref{bound.ll}), we obtain
\begin{eqnarray}
&&-\int_\mathcal M\partial_tq_r\Delta_hq_rd\mathcal M\nonumber\\
&=&
-\int_{\Gamma_\ell}\partial_tq_r\partial_{\mathbf{n}}q_rd\Gamma_\ell +
\int_\mathcal M\partial_t\nabla_hq_r\cdot\nabla_hq_rd\mathcal M\nonumber\\
&=&\frac12\frac{d}{dt}\|\nabla_hq_r\|^2-\alpha_{\ell v}
\int_{\Gamma_\ell}\partial_tq_r(q_{b\ell r}-q_r)d\Gamma_\ell\nonumber \\
&=&\frac{d}{dt}\left(\frac{\|\nabla_hq_r\|^2}{2}+\alpha_{\ell r}\int_{
\Gamma_\ell}\left(\frac{q_r^2}{2}-q_rq_{b\ell r}\right)d\Gamma_\ell\right)
 +\alpha_{\ell r}\int_{\Gamma_\ell}q_r\partial_tq_{b\ell r}d\Gamma_\ell
\nonumber\\
&\geq&\frac{d}{dt}\left(\frac{\|\nabla_hq_r\|^2}{2}+\alpha_{\ell r}
\int_{\Gamma_\ell}\left(\frac{q_r^2}{2}-q_rq_{b\ell r}\right)d\Gamma_\ell
\right)-C(1+\|\nabla_hq_r\|),
\label{thqr}
\end{eqnarray}
where in the last step we have used
\begin{eqnarray*}
\left|\int_{\Gamma_\ell}q_r\partial_tq_{b\ell r}d\Gamma_\ell\right|
\leq C\|q_r\|_{L^1(\Gamma_\ell)}\leq C(\|q_r\|
+\|\nabla_hq_r\|)\leq C(1+\|\nabla_hq_r\|),
\end{eqnarray*}
which is guaranteed by the trace inequality and Proposition \ref{bddq}.

Recalling the expression of $\mathcal D^{q_r}q_r$, we have
\begin{equation}
  \int_\mathcal M\mathcal D^{q_r}q_r\Delta_hq_vd\mathcal M
=\mu_{q_r}\|\Delta_hq_r\|^2+\nu_{q_r}\int_\mathcal M\partial_p\left(
\left(\frac{gp}{R_d\bar T}\right)^2\partial_pq_r\right)\Delta_hq_vd\mathcal M.
\label{df.hqr.0}
\end{equation}
Moreover integration by parts yields
\begin{eqnarray}
&&\int_\mathcal M\partial_p\left(
\left(\frac{gp}{R_d\bar T}\right)^2\partial_pq_r\right)\Delta_hq_vd\mathcal M\nonumber\\
&=&\left.\int_{\mathcal M'}\left(\frac{gp}{R_d\bar T}\right)^2\partial_p q_r
\Delta_hq_rd\mathcal M'\right|_{p_0}
-\int_\mathcal M\left(\frac{gp}{R_d\bar T}\right)^2\partial_pq_r\partial_p\Delta_hq_rd\mathcal M\nonumber\\
&=&\left.\int_{\mathcal M'}\left(\frac{gp}{R_d\bar T}\right)^2\partial_p q_r
\Delta_hq_rd\mathcal M'\right|_{p_0}
+\|\nabla_h\partial_pq_r\|_w^2\nonumber\\
&&-\int_{\Gamma_\ell}\left(\frac{gp}{R_d\bar T}\right)^2\partial_pq_r\partial_p\partial_{\mathbf{n}}q_rd\Gamma_\ell.
\label{df.hqr.1}
\end{eqnarray}
We denote hereafter $\delta_0=\alpha_{0r}\left(\frac{gp_0}{R_d\bar T(p_0)}\right)^2$. Then, using the boundary condition (\ref{bound.0}) and \eqref{est.qbl}--\eqref{est.qblgrad}, it follows from integration by parts that
\begin{eqnarray}
  &&\left.\int_{\mathcal M'}\left(\frac{gp}{R_d\bar T}\right)^2\partial_p q_r
\Delta_hq_rd\mathcal M'\right|_{p_0}=\delta_0
\left.\int_{\mathcal M'}(q_{b0r}-q_r)\Delta_hq_rd\mathcal M'\right|_{p_0}\nonumber\\
&=&\delta_0\left.\left(\int_{\partial\mathcal M'}(q_{b0r}-q_r)\partial_{\mathbf{n}}q_rd\partial\mathcal M'+\int_{\mathcal M'}\nabla_h(q_r-q_{b0r})\cdot \nabla_hq_r d\mathcal M'\right)\right|_{p_0}\nonumber\\
&=&\delta_0
\left.\left(\alpha_{\ell r}\int_{\partial\mathcal M'}(q_{b\ell r}-q_r)(q_{b0r}-q_r)d\partial\mathcal M'+
\int_{\mathcal M'}\nabla_hq_r\cdot\nabla_h(q_r-q_{b0r})d\mathcal M\right)\right|_{p_0}\nonumber\\
&\geq&-\frac{\delta_0}{4}
\left.\left(\alpha_{\ell r}\int_{\partial\mathcal M'}(q_{b\ell r}-q_{b0r})^2d\partial\mathcal M'+\int_{\mathcal M'}|\nabla_hq_{b0r}|^2d\mathcal M'\right)\right|_{p_0}\geq-C,
\label{df.hqr.2}
\end{eqnarray}
where we note that here and in the subsequent estimates the constant $C$ depends on the initial and boundary data.
Using the boundary condition (\ref{bound.ll}), we have moreover
\begin{eqnarray}
&&-\int_{\Gamma_\ell}\left(\frac{gp}{R_d\bar T}\right)^2\partial_pq_r\partial_p\partial_{\mathbf{n}}q_rd\Gamma_\ell
\nonumber\\
&=&\alpha_{\ell r}\int_{\Gamma_\ell}\left(\frac{gp}{R_d\bar T}\right)^2\partial_pq_r (\partial_pq_r-\partial_pq_{b\ell r})d\Gamma_\ell\nonumber\\
&=&\alpha_{\ell r}\int_{\Gamma_\ell}\left(\frac{gp}{R_d\bar T}\right)^2\left[\left(\partial_pq_r-\frac{\partial_pq_{b\ell r}}{2}\right)^2-\frac{|\partial_pq_{b\ell r}|^2}{4}\right]d\Gamma_\ell
\nonumber\\
&\geq&-\frac{\alpha_{\ell r}}{4}\int_{\Gamma_\ell} \left(\frac{gp}{R_d\bar T}\right)^2(\partial_pq_{b\ell r})^2d\Gamma_\ell\geq-C.\label{df.hqr.3}
\end{eqnarray}
Substituting (\ref{df.hqr.2}) and (\ref{df.hqr.3}) into (\ref{df.hqr.1}) yields
\begin{equation*}
  \int_\mathcal M\partial_p\left(
\left(\frac{gp}{R_d\bar T}\right)^2\partial_pq_r\right)\Delta_hq_vd\mathcal M\geq
\|\nabla_h\partial_pq_r\|_w^2-C,
\end{equation*}
and, consequently, it follows from (\ref{df.hqr.0}) that
\begin{equation}
  \label{df.hqr}
\int_\mathcal M\mathcal D^{q_r}q_r\Delta_hq_vd\mathcal M
\geq\mu_{q_r}\|\Delta_hq_r\|^2+\nu_{q_r}\|\nabla_h\partial_pq_r\|_w^2
-C.
\end{equation}
Note that
\begin{eqnarray*}
|u(x,y,p)|&=&\left|\frac{1}{p_0-p_1}\int_{p_1}^{p_0}ud\sigma+ \frac{1}{p_0-p_1}\int_{p_1}^{p_0}\left(\int_\sigma^p\partial_pud\xi\right) d\sigma\right|\\
&\leq&\frac{1}{p_0-p_1}\int_{p_1}^{p_0}|u|d\sigma+\int_{p_1}^{p_0}|\partial_p u|d\sigma,
\end{eqnarray*}
for any $p\in[p_1,p_0]$. Thanks to this, it follows from Lemma \ref{lemlad}
and \corr{Young's inequality} that
\begin{eqnarray}
  &&\int_\mathcal M(\mathbf{u}\cdot\nabla_hq_r+\omega\partial_pq_r)
\Delta_hq_rd\mathcal M\nonumber\\
&\leq&C\int_{\mathcal M'}\int_{p_1}^{p_0}(|u|+|\partial_pu|)dp\int_{p_1}^{p_0}|\nabla_hq_r|
|\Delta_hq_r|dpd\mathcal M'\nonumber\\
&&+\int_{\mathcal M'}\int_{p_1}^{p_0}|\nabla_h\mathbf{u}|dp\int_{p_1}^{p_0}|\partial_pq_r| |\Delta_hq_r|dpd\mathcal M'\nonumber\\
&\leq&C\left[\|\mathbf{u}\|^{\frac12}\left(\|\mathbf{u}\|^{\frac12}
+\|\nabla_h\mathbf{u}\|^{\frac12}\right)+\|\partial_p\mathbf{u}\|^{\frac12}
\left(\|\partial_p\mathbf{u}\|^{\frac12}+\|\nabla_h\partial_p\mathbf{u}\|
^{\frac12}\right)\right]\|\nabla_hq_r\|^{\frac12}\nonumber\\
&&\times\left(\|\nabla_hq_r\|^{\frac12} +\|\nabla_h^2q_r\|^{\frac12}\right)\|\Delta_hq_r\| +C\|\nabla_h\mathbf{u}\|^{\frac12}\left(\|\nabla_h^2\mathbf{u}
\|^{\frac12}+\|\nabla_h\mathbf{u}\|^{\frac12}\right)\nonumber\\
&&\times\|\partial_pq_r
\|^{\frac12}\left(\|\partial_pq_r\|^{\frac12}+\|\nabla_h\partial_pq_r\|
^{\frac12}\right)\|\Delta_hq_r\|\nonumber\\
&\leq&C\left[\left(1
+\|\nabla_h\mathbf{u}\|^{\frac12}\right)+
\left(1+\|\nabla_h\partial_p\mathbf{u}\|
^{\frac12}\right)\right]\|\nabla_hq_r\|^{\frac12} \left(\|\nabla_hq_r\|^{\frac12}
+\|\nabla_h^2q_r\|^{\frac12}\right)\nonumber\\
&&\times
\|\Delta_hq_r\|+C\|\nabla_h\mathbf{u}\|^{\frac12}\left(\|\nabla_h^2\mathbf{u}
\|^{\frac12}+\|\nabla_h\mathbf{u}\|^{\frac12}\right) \left(1+\|\nabla_h\partial_pq_r\|
^{\frac12}\right)\|\Delta_hq_r\|,\label{covhq.0}
\end{eqnarray}
where, we have used Proposition \ref{bddq}, Proposition \ref{est.u.pu}, and Proposition \ref{prop.est.pq}.
By \corr{standard} elliptic estimates and Proposition \ref{bddq}, we have
\begin{equation}\label{elliptic}
\|\nabla^2_h\mathbf{u}\|\leq C\|\Delta_h\mathbf{u}\|,\quad \|\nabla_h^2 q_i\|\leq C(\|\Delta_hq_i\|+\|q_i\|+1)\leq C(\|\Delta_hq_i\|+1),
\end{equation}
for $i\in\{v,c,r\}$.
Then it follows from (\ref{covhq.0}) and \corr{Young's inequality} that
\begin{eqnarray}
  &&\int_\mathcal M(\mathbf{u}\cdot\nabla_hq_r+\omega\partial_pq_r)
\Delta_hq_rd\mathcal M\nonumber\\
&\leq&C\left[\left(1
+\|\nabla_h\mathbf{u}\|^{\frac12}\right)+
\left(1+\|\nabla_h\partial_p\mathbf{u}\|
^{\frac12}\right)\right]\|\nabla_hq_r\|^{\frac12} \left(1+\|\nabla_hq_r\|^{\frac12}
+\|\Delta_hq_r\|^{\frac12}\right)\nonumber\\
&&
\times\|\Delta_hq_r\|+C\|\nabla_h\mathbf{u}\|^{\frac12}\left(1+\|\Delta_h\mathbf{u}
\|^{\frac12}+\|\nabla_h\mathbf{u}\|^{\frac12}\right) \left(1+\|\nabla_h\partial_pq_r\|
^{\frac12}\right)\|\Delta_hq_r\|\nonumber\\
&\leq&\frac{\mu_{q_r}}{8}\|\Delta_hq_r\|^2+\frac{\mu_{\mathbf{u}}}{12}
\|\Delta_h\mathbf{u}\|^2+C(1+\|\nabla_h\mathbf{u}\|^2+\|\nabla_h\partial_p
\mathbf{u}\|^2)(1+\|\nabla_hq_r\|^2)\nonumber\\
&&+C(1+\|\nabla_h\partial_pq_r\|^2)(1+\|\nabla_h\mathbf{u}\|^2).
\label{covhq}
\end{eqnarray}

Thanks to (\ref{est.hqr.1}), (\ref{thqr}), (\ref{df.hqr}), and (\ref{covhq}), it follows from (\ref{est.hqr.0}) that
\begin{align}
  \frac{d}{dt}\left(\frac{\|\nabla_hq_r\|^2}{2}\right.
+\alpha_{\ell r}&\left.\int_{\Gamma_\ell}
\left(\frac{q_r^2}{2}-q_rq_{b\ell r}\right)d
\Gamma_\ell\right)+\frac{3\mu_{q_r}}{4}\|\Delta_hq_r\|^2
+\nu_{q_r}\|\nabla_h\partial_pq_r\|_w^2\nonumber\\
\leq& \ \ \frac{\mu_{\mathbf{u}}}{12}
\|\Delta_h\mathbf{u}\|^2+C(1+\|\nabla_h\mathbf{u}\|^2
+\|\nabla_h\partial_p
\mathbf{u}\|^2)(1+\|\nabla_hq_r\|^2)\nonumber\\
&\ \ +C(1+\|\nabla_h\partial_pq_r\|^2)(1
+\|\nabla_h\mathbf{u}\|^2). \nonumber
\end{align}
The same estimate as above also holds for $q_v, q_c$, such that
\begin{align}
  \frac{d}{dt}\left(\frac{\|\nabla_hq_j\|^2}{2}\right.+\alpha_{\ell j}
&\left.\int_{\Gamma_\ell}\left(\frac{q_j^2}{2}-q_jq_{b\ell j}\right)d
\Gamma_\ell\right)+\frac{3\mu_{q_j}}{4}\|\Delta_hq_j\|^2
+\nu_{q_j}\|\nabla_h\partial_pq_j\|_w^2\nonumber\\
\leq& \ \ \frac{\mu_{\mathbf{u}}}{12}
\|\Delta_h\mathbf{u}\|^2+C(1+\|\nabla_h\mathbf{u}\|^2+\|\nabla_h
\partial_p
\mathbf{u}\|^2)(1+\|\nabla_hq_j\|^2)\nonumber\\
&\ \ +C(1+\|\nabla_h\partial_pq_j\|^2)(1+\|\nabla_h\mathbf{u}\|^2),
\label{est.hq}
\end{align}
for $j\in\{v,c,r\}$.

Summing (\ref{est.u.hu}) with (\ref{est.hq}) for $j\in\{v,c,r\}$ yields
\begin{eqnarray}
&&\frac{d}{dt}
\left[\|\nabla_h\mathbf{u}\|^2+\sum_{j\in\{v,c,r\}}
\left(\frac{\|\nabla_hq_j\|^2}{2}
+\alpha_{\ell j}
\int_{\Gamma_\ell}\left(\frac{q_j^2}{2}-q_jq_{b\ell j}
\right)d
\Gamma_\ell\right)\right]\nonumber\\
&&+\sum_{j\in\{v,c,r\}}\left(\frac{\mu_{q_j}}{2}\|\Delta_hq_j\|^2
+\nu_{q_j}\|\nabla_h\partial_pq_j\|_w^2\right)+\frac{\mu_\mathbf{u}}{2}
\|\Delta_h\mathbf{u}\|^2 + \nu_\mathbf{u}\|\pa_p\nabla_h
\mathbf{u}\|^2
\nonumber\\
&\leq&C\left(1+\|\nabla_h\mathbf{u}\|^2+\|\nabla_hT\|^2+\|\nabla_h
\partial_p
\mathbf{u}\|^2+\sum_{j\in\{v,c,r\}}\|\nabla_h\partial_pq_j\|^2\right)
\nonumber\\
&&\times\left(1+\|\nabla_h\mathbf{u}\|^2+\sum_{j\in\{v,c,r\}}
\|\nabla_hq_j\|^2\right),\nonumber
\end{eqnarray}
from which, noticing that
$$
\int_{\Gamma_\ell}\left(\frac{q_j^2}{2}-q_jq_{b\ell j}
\right)d
\Gamma_\ell=\frac12\int_{\Gamma_\ell}\left((q_j-q_{b\ell j})^2-q_{b\ell j}^2
\right)d
\Gamma_\ell\geq
-\frac12\int_{\Gamma_\ell}q_{b\ell j}^2
d\Gamma_\ell\geq-C,
$$
the conclusion follows by the Gronwall inequality and
Propositions \ref{basic}--\ref{prop.est.pq}.
\end{proof}

Finally, we have the control of the gradient of the temperature.

\begin{prop}
  \label{prop.est.dt}
Let Assumption \ref{ass} hold, then there exists a function $K_6(t)$, which depends on the initial and boundary data, and is continuous for all $t\geq 0$, such that
\begin{equation*}
  \sup_{0\leq t\leq\mathcal T}\|\nabla_h T\|^2+\int_0^\mathcal T(\|\nabla_h^2T\|^2+\|\nabla_h \pa_p T\|_w^2)dt\leq K_6(\mathcal T),
\end{equation*}
and
\begin{equation*}
  \sup_{0\leq t\leq\mathcal T}\|\pa_p T\|^2+\int_0^\mathcal T(\|\nabla_h\pa_pT\|^2+\| \pa^2_p T\|_w^2)dt\leq K_6(\mathcal T),
\end{equation*}
for any $\mathcal T\in[0,\mathcal T_{\text{max}})$.
\end{prop}

\begin{proof}
We first estimate the vertical derivative $\partial_pT$. Multiplying the thermodynamic equation \corr{\eqref{AT}} by $-\partial_p^2T$ and integrating the resultant over $\mathcal M$ yields
\begin{eqnarray}
&&\int_\mathcal M(-\partial_tT+\mathcal D^TT)\partial_p^2Td\mathcal M
\nonumber\\
&=&\int_\mathcal M\left[\mathbf{u}\cdot\nabla_hT+\omega\partial_pT
  -\frac{R}{c_p}\frac{T}{p}\omega+\frac{L}{c_p}(S_{cd}-S_{ev,\varepsilon})\right]
  \partial_p^2Td\mathcal M.\label{est.pt.0}
\end{eqnarray}
Following the derivation of (\ref{tpqr}), (\ref{df.npqr}), and (\ref{covpq}),
we obtain
\begin{equation}
-\int_\mathcal M\partial_tT\partial_p^2Td\mathcal M
\geq\frac{d}{dt}\left(\frac{\|\partial_pT\|^2}{2}+\alpha_{0T}
\left.\int_{\mathcal M'}\left(\frac{T^2}{2}-TT_{b0}\right) d\mathcal M'\right|_{p_0}\right)-C(\|\partial_pT\|+1), \label{est.pt.1}
\end{equation}
\begin{equation}
\int_\mathcal M\mathcal D^{T}T\partial_p^2Td\mathcal M
\geq \frac34\left(\mu_{T}\|\nabla_h\partial_pT\|^2+\nu_{T}
\|\partial_p^2T\|_w^2\right)-C(\|\partial_pT\|^2+1),\label{est.pt.2}
\end{equation}
and
\begin{eqnarray}
  &&\int_\mathcal M(\mathbf{u}\cdot\nabla_hT+\omega\partial_pT)
\partial_p^2Td\mathcal M\nonumber\\
&\leq&\frac{\mu_{T}}{12}\|\nabla_h\partial_pT\|^2+C(\|\nabla_hT\|^2+\|\nabla_h\partial_p\mathbf{u}\|^2+1)
(1+\|\partial_pT\|^2).\label{est.pt.3}
\end{eqnarray}
By Lemma \ref{lemlad}, Proposition \ref{basic}, Proposition \ref{L2T}, Proposition \ref{prop.est.hu}, \corr{Young's inequality}, and recalling (\ref{elliptic}), we deduce
\begin{eqnarray}
-\int_\mathcal M\frac{R T}{c_pp}\omega\partial_p^2Td\mathcal M
&\leq&C\int_{\mathcal M'}\int_{p_1}^{p_0}|\nabla_h\mathbf{u}|dp
\int_{p_1}^{p_0}|T||\partial_p^2T|dpd\mathcal M\nonumber\\
&\leq&C\|\nabla_h\mathbf{u}\|^{\frac12}\|\nabla_h^2\mathbf{u}\|^{\frac12}
\|T\|^{\frac12}\left(\|T\|^{\frac12}+\|\nabla_hT\|^{\frac12}\right)
\|\partial_p^2T\|\nonumber\\
&\leq&\frac{\mu_T}{12}\|\partial_p^2T\|^2+C(\|\Delta_h\mathbf{u}\|^2 +\|\nabla_hT\|^2+1).\label{est.pt.4}
\end{eqnarray}
Recalling  (\ref{bddsource}), it follows from \corr{Young's inequality} and Proposition \ref{L2T} that
\begin{equation}
  \int_\mathcal M\frac{L}{c_p}(S_{cd}-S_{ev,\varepsilon})\partial_p^2Td\mathcal M
  \leq C\int_\mathcal M|\partial_p^2T|d\mathcal M
  \leq\frac{\mu_T}{12}\|\partial_p^2T\|^2+C.\label{est.pt.5}
\end{equation}
Substituting (\ref{est.pt.1})--(\ref{est.pt.5}) into (\ref{est.pt.0}), and recalling Proposition \ref{basic} and Proposition \ref{L2T}, one obtains
\begin{eqnarray*}
\frac{d}{dt}\left(\frac{\|\partial_pT\|^2}{2}+\alpha_{0T}
\left.\int_{\mathcal M'}\left(\frac{T^2}{2}-TT_{b0}\right) d\mathcal M'\right|_{p_0}\right)+\left(\frac{\mu_{T}}{2}\|\nabla_h\partial_pT\|^2+
\frac{\nu_{T}}{2}
\|\partial_p^2T\|_w^2\right)\nonumber\\
\leq C(\|\nabla_hT\|^2+\|\Delta_h\mathbf{u}\|^2+\|\nabla_h\partial_p\mathbf{u}\|^2+1)
(1+\|\partial_pT\|^2),
\end{eqnarray*}
from which, noticing that
\begin{eqnarray*}
\left.\int_{\mathcal M'}(T^2-2TT_{b0})d\mathcal M'\right|_{p_0}
=\left.\int_{\mathcal M'}\left[(T-T_{b0})^2-T_{b0}^2\right]d\mathcal M'\right|_{p_0}
\geq-\left.\int_{\mathcal M'}T_{b0}^2 d\mathcal M'\right|_{p_0}\geq-C,
\end{eqnarray*}
by the Gronwall inequality,
Proposition \ref{L2T}, Proposition \ref{est.u.pu}, and Proposition \ref{prop.est.hu}, we obtain
\begin{equation}
  \sup_{0\leq t\leq \mathcal T}\|\partial_pT\|^2+\frac12\int_0^\mathcal T(\mu_T\|\nabla_h\partial_pT\|^2+\nu_T\|\partial_p^2T\|_w^2)dt\leq K_6'(\mathcal T). \label{est.pt}
\end{equation}

Next, we estimate the horizontal gradient $\nabla_hT$. Multiplying equation (\ref{AT}) by $-\Delta_hT$ and integrating the resultant over $\mathcal M$ yield
\begin{eqnarray}
  &&\int_\mathcal M(-\partial_tT+\mathcal D^TT)\Delta_hTd\mathcal M
\nonumber\\
&=&\int_\mathcal M\left[\mathbf{u}\cdot\nabla_HT+\omega\partial_pT
  -\frac{ R}{c_p}\frac{T}{p}\omega+\frac{L}{c_p}(S_{cd}-S_{ev,\varepsilon})\right]
  \Delta_hTd\mathcal M.\label{est.ht.0}
\end{eqnarray}
Following the derivations in (\ref{thqr}) and (\ref{df.hqr}), we obtain
\begin{eqnarray}
-\int_\mathcal M\partial_tT\Delta_hTd\mathcal M&\geq&\frac{d}{dt}\left(\frac{\|\nabla_hT\|^2}{2}+\alpha_{\ell T}
\int_{\Gamma_\ell}\left(\frac{T^2}{2}-TT_{b\ell}\right)d\Gamma_\ell
\right)\nonumber\\
&&-C(1+\|\nabla_hT\|),\label{est.ht.1}
\end{eqnarray}
and
\begin{equation}
\int_\mathcal M\mathcal D^TT\Delta_hTd\mathcal M
\geq\mu_T\|\Delta_hT\|^2+\nu_T\|\nabla_h\partial_pT\|_w^2
-C.\label{est.ht.2}
\end{equation}
Similar to (\ref{covhq}), we get
\begin{eqnarray}
&&\int_\mathcal M(\mathbf{u}\cdot\nabla_hT+\omega\partial_pT)
\Delta_hTd\mathcal M\nonumber\\
&\leq&\frac{\mu_T}{6}\|\Delta_hT\|^2+\frac{\mu_{\mathbf{u}}}{12}
\|\Delta_h\mathbf{u}\|^2+C(1+\|\nabla_h\mathbf{u}\|^2+\|\nabla_h\partial_p
\mathbf{u}\|^2)(1+\|\nabla_hT\|^2)\nonumber\\
&&+C(1+\|\nabla_h\partial_pT\|^2)(1+\|\nabla_h\mathbf{u}\|^2),\nonumber
\end{eqnarray}
from which, by Proposition \ref{prop.est.hu}, one obtains
\begin{eqnarray}
\int_\mathcal M(\mathbf{u}\cdot\nabla_hT+\omega\partial_pT)
\Delta_hTd\mathcal M
&\leq&C(1+\|\nabla_h\partial_pT\|^2+
\|\nabla_h\partial_p
\mathbf{u}\|^2)(1+\|\nabla_hT\|^2) \nonumber\\
&&+\frac{\mu_T}{6}\|\Delta_hT\|^2+\frac{\mu_{\mathbf{u}}}{12}
\|\Delta_h\mathbf{u}\|^2.
\label{est.ht.3}
\end{eqnarray}
Following the derivations in (\ref{est.pt.4}) and (\ref{est.pt.5}), we get
\begin{equation}
-\int_\mathcal M\frac{R T}{c_pp}\omega\Delta_hTd\mathcal M
\leq\frac{\mu_T}{6}\|\Delta_hT\|^2+C(\|\Delta_h\mathbf{u}\|^2 +\|\nabla_hT\|^2+1).\label{est.ht.4}
\end{equation}
and
\begin{equation}
  \int_\mathcal M\frac{L}{c_p}(S_{cd}-S_{ev,\varepsilon})\Delta_hTd\mathcal M
  \leq  \frac{\mu_T}{6}\|\Delta_hT\|^2+C.\label{est.ht.5}
\end{equation}
Substituting (\ref{est.ht.1})--(\ref{est.ht.5}) into (\ref{est.ht.0}) gives
\begin{eqnarray*}
  \frac{d}{dt}\left(\frac{\|\nabla_hT\|^2}{2}+\alpha_{\ell T}
\int_{\Gamma_\ell}\left(\frac{T^2}{2}-TT_{b\ell}\right)d\Gamma_\ell
\right)+\frac12(\mu_T\|\Delta_hT\|^2+\nu_T\|\nabla_h\partial_pT\|_w^2)\\
\leq C(1+\|\nabla_h\partial_pT\|^2+\|\Delta_h\mathbf{u}\|^2+
\|\nabla_h\partial_p
\mathbf{u}\|^2)(1+\|\nabla_hT\|^2),
\end{eqnarray*}
from which, noticing that
\begin{eqnarray*}
\int_{\Gamma_\ell}\left(\frac{T^2}{2}-TT_{b\ell}
\right)d
\Gamma_\ell&=&\frac12\int_{\Gamma_\ell}\left((T-T_{b\ell})^2-T_{b\ell}^2
\right)d
\Gamma_\ell\geq
-\frac12\int_{\Gamma_\ell}T_{b\ell}^2
d\Gamma_\ell\geq-C,
\end{eqnarray*}
it follows from the Gronwall inequality, Proposition \ref{est.u.pu}, Proposition \ref{prop.est.hu}, (\ref{est.pt}), and (\ref{elliptic}) that
\begin{equation*}
  \sup_{0\leq t\leq\mathcal T}\|\nabla_hT\|^2+\int_0^\mathcal T
  (\|\nabla_h^2T\|^2+\|\nabla_h\partial_pT\|_w^2)dt\leq K_6''(\mathcal T),
\end{equation*}
where $K_6''(t)$ depends on the initial and boundary data, and is continuous for all $t\geq 0$. Combining this with (\ref{est.pt}) yields the conclusion.
\end{proof}

As a corollary of Propositions \ref{Aloc}--\ref{prop.est.dt}, we have the following global existence
and a priori estimates for system (\ref{Au.1})--(\ref{Aqr}), subject to (\ref{bound.0})--(\ref{IC}).

\begin{cor}
  \label{COR}
Assume that $\mathbf{u}_0, T_0, q_{v0}, q_{c0}, q_{r0}\in H^1
(\mathcal M)$ and $T_0,q_{v0}, q_{c0}, q_{r0} \in L^\infty(\mathcal M)$, with $T_0$, $q_{v0}$, $q_{c0}$,
$q_{r0}\geq0$ on $\mathcal M$ and
$\int_{p_0}^{p_1}\nabla_h\cdot\mathbf{u}_0dp=0$ on $\mathcal M'$. Then,
system (\ref{Au.1})--(\ref{Aqr}), subject to (\ref{bound.0})--(\ref{IC}) has a unique global strong solution $(\mathbf{u}, T, q_v, q_c, q_r)$, satisfying
\begin{eqnarray*}
&T, q_v, q_c, q_r\geq0, \quad \textnormal{and}\quad T,q_v,q_c,q_r \in L^\infty(0,\mathcal T; L^\infty),\\
&\mathbf{u}, T, q_v, q_c, q_r  \in C([0,\mathcal T]; H^1(\mathcal M))\cap L^2(0,\mathcal T; H^2(\mathcal M)),
\\
&\partial_t\mathbf{u},\partial_tT, \partial_tq_v, \partial_tq_c, \partial_tq_r \in L^2(0,\mathcal
T; L^2(\mathcal M)),
\end{eqnarray*}
and the a priori estimate
\begin{equation*}
  \sup_{0\leq t\leq\mathcal T}(\|(T,q_v, q_c, q_r)\|_{L^\infty(\mathcal M)}+\|(\mathbf{u}, T, q_v, q_c, q_r)\|_{H^1(\mathcal M)})
  \leq
  K(\mathcal T),
\end{equation*}
and
\begin{equation*}
\int_0^\mathcal T(\|(\mathbf{u}, q_v, q_c, q_r)\|_{H^2(\mathcal
  M)}^2+\|(\partial_t\mathbf{u},\partial_tT,\partial_tq_v,
  \partial_tq_c,\partial_tq_r)\|^2)dt\leq
  K(\mathcal T),
\end{equation*}
for any $\mathcal T\in(0,\infty)$, where $K$ is a continuous function on $[0,\infty)$, depending
only on the initial and boundary data,  which is
independent of $\varepsilon\in(0,1)$.
\end{cor}

\begin{proof}
We need to prove $\mathcal T_{\text{max}}=\infty$. Assume, by contradiction, that $\mathcal T_{\text{max}}<\infty$. By Propositions \ref{bddq}--\ref{prop.est.dt}, we have the estimate
\begin{align*}
  \sup_{0\leq t\leq\mathcal T}
  \|(\mathbf{u}, T,q_v, q_c, q_r)\|_{H^1(\mathcal M)}\leq C_0,
\end{align*}
for any $\mathcal T\in(0,\mathcal T_{\text{max}})$, and $C_0$ is a positive constant depending on the initial and boundary data, but which is independent of
$\mathcal T\in(0,\mathcal T_{\text{max}})$. This contradicts (\ref{T_max}) and, thus,
$\T_{\text{max}}=\infty$. The a priori estimates except those involving the time derivatives follow
directly from Propositions \ref{bddq}--\ref{prop.est.dt}; while the desired estimates for the time
derivative follow from those in Propositions \ref{bddq}--\ref{prop.est.dt}, using equations
(\ref{Au.1}), (\ref{AT})--(\ref{Aqr}). The proof is lengthy but standard, and therefore omitted. \end{proof}

\section{Global existence and uniqueness}
\label{sec.pf}

We are now ready to prove the global existence and uniqueness result, i.e., Theorem \ref{thm}.

\begin{proof}[Proof of Theorem \ref{thm}]
\textbf{(i) Existence. }By Corollary \ref{COR} for any $\varepsilon\in(0,1)$, there is a unique
global solution $(\mathbf{u}_\varepsilon, T_\varepsilon, q_{v\varepsilon}, q_{c\varepsilon}, q_{r\varepsilon})$ satisfying
\begin{equation}
  T_\varepsilon, q_{v\varepsilon}, q_{c\varepsilon}, q_{r\varepsilon}\geq0, \label{AST0}
\end{equation}
and the
a priori estimates
 \begin{equation}
  \sup_{0\leq t\leq\mathcal T}(\|(T_\varepsilon,q_{v\varepsilon}, q_{c\varepsilon}, q_{r\varepsilon})\|_{L^\infty(\mathcal M)}+
  \|(\mathbf{u}_\varepsilon, T_\varepsilon, q_{v\varepsilon}, q_{c\varepsilon}, q_{r\varepsilon})\|_{H^1(\mathcal M)})\leq
  K(\mathcal T),\label{AST1}
 \end{equation}
 and
 \begin{equation}
  \int_0^\mathcal T(\|(\mathbf{u}_\varepsilon, T_\varepsilon, q_{v\varepsilon}, q_{c\varepsilon}, q_{r\varepsilon})
  \|_{H^2(\mathcal
  M)}^2+\|(\partial_t\mathbf{u_\varepsilon},
  \partial_tT_\varepsilon,\partial_tq_{v\varepsilon},
  \partial_tq_{c\varepsilon},\partial_tq_{r\varepsilon})\|^2)dt\leq
  K(\mathcal T),\label{AST2}
\end{equation}
for any $\mathcal T\in(0,\infty)$, for a continuous function $K$ on $[0,\infty)$ independent of $\varepsilon$. Thanks to the a priori estimates (\ref{AST1}) and (\ref{AST2}), by the Banach-Alaoglu theorem,
and using Cantor's diagonal argument to $\varepsilon$, there is a subsequence (still denoted by $(\mathbf{u}_\varepsilon, T_\varepsilon, q_{v\varepsilon}, q_{c\varepsilon}, q_{r\varepsilon})$), and $(\mathbf{u}, T, q_v, q_c, q_r)$, such that
\begin{eqnarray}
  &&(\mathbf{u}_\varepsilon, T_\varepsilon, q_{v\varepsilon}, q_{c\varepsilon}, q_{r\varepsilon})\rightharpoonup^*(\mathbf{u}, T, q_v, q_c, q_r)\quad\mbox{in }L^\infty(0,\mathcal T; H^1(\mathcal M)), \label{lim1}\\
  &&(\mathbf{u}_\varepsilon, T_\varepsilon, q_{v\varepsilon}, q_{c\varepsilon}, q_{r\varepsilon})\rightharpoonup(\mathbf{u}, T, q_v, q_c, q_r)\quad\mbox{in }L^2(0,\mathcal T; H^2(\mathcal M)),\label{lim2}
\end{eqnarray}
and
\begin{eqnarray}
  (\partial_t\mathbf{u}_\varepsilon, \partial_tT_\varepsilon, q_{v\varepsilon}, \partial_tq_{c\varepsilon}, \partial_tq_{r\varepsilon})\rightharpoonup(\partial_t\mathbf{u},\partial_t
   T,\partial_t q_v,\partial_t q_c,\partial_t q_r)\quad\mbox{in }L^2(0,\mathcal T; L^2(\mathcal M)),\label{lim3}
\end{eqnarray}
where $\rightharpoonup$ and $\rightharpoonup^*$, respectively, denote the weak and weak-* convergences in the corresponding spaces. By the Aubin-Lions Compactness Lemma (see e.g. \cite{Si}), it follows from (\ref{lim1})--(\ref{lim3}) that
\begin{equation}
  (\mathbf{u}_\varepsilon, T_\varepsilon, q_{v\varepsilon}, q_{c\varepsilon}, q_{r\varepsilon})\rightarrow(\mathbf{u}, T, q_v, q_c, q_r)\quad\mbox{in }C([0,\mathcal T]; L^2(\mathcal M))\cap L^2(0,\mathcal T; H^1(\mathcal M)), \label{lim4}
\end{equation}
from which, recalling (\ref{AST0}) and (\ref{AST1}), one has
\begin{equation}
  (T,q_v, q_c, q_r)\in L^\infty(0,\T; L^\infty(\mathcal M)),\quad T, q_v, q_c, q_r\geq0.
\end{equation}
Thanks to the convergences (\ref{lim1})--(\ref{lim4}), we can take the limit as $\varepsilon\rightarrow0$ to show that $(\mathbf{u}, T, q_v, q_c, q_r)$ is a solution to system \eqref{eq.qv}--\eqref{eq.T}, \eqref{eq.u}--\eqref{eq.Phi}, subject to (\ref{bound.0})--(\ref{IC}).

\textbf{(ii) Uniqueness. }Let $(\mathbf{u}_i, T_i, q_{vi}, q_{ci}, q_{ri})$, $i=1,2$, be two solutions, and denote
by $(\mathbf{u}, T, q_v, q_c, q_r)$ their difference. Then $\mathbf{u}$ satisfies
\begin{eqnarray*}
  \partial_t\mathbf{u}+(\mathbf{u_1}\cdot\nabla_h)\mathbf{u}+\omega_1\partial_p\mathbf{u}
  +f(k\times\mathbf{v})_h-\mathcal D^{\mathbf{u}}\mathbf{u}
  =-(\mathbf{u}\cdot\nabla_h)\mathbf{u}_2-\omega\partial_p\mathbf{u}_2-\nabla_h \Phi,
\end{eqnarray*}
where $\Phi=\Phi_1-\Phi_2$.
Multiplying the above equation by $\mathbf{u}$ and integrating over $\mathcal M$ yields
\begin{eqnarray}
  \frac12\frac{d}{dt}\|\mathbf{u}\|^2-\int_\mathcal M\mathcal D^{\mathbf{u}}\mathbf{u}\cdot\mathbf{u}d\mathcal M=-\int_\mathcal M[(\mathbf{u}\cdot\nabla_h)\mathbf{u}_2+\omega\partial_p\mathbf{u}_2+\nabla_h\Phi]
  \cdot\mathbf{u}d\mathcal M.\label{uni.u.0}
\end{eqnarray}
As in (\ref{egyu.2}), we have
\begin{equation}
  -\int_\mathcal M\mathcal D^{\mathbf{u}}\mathbf{u}\cdot\mathbf{u}d\mathcal M
  \geq\mu_{\mathbf{u}}\|\nabla_h\mathbf{u}\|^2+\nu_{\mathbf{u}}\|\partial_pu\|_w^2.
  \label{uni.u.1}
\end{equation}
By (\ref{eq.Phi}),  the boundary condition (\ref{bound.ll}) for $\mathbf{u}$,
and noticing that $\int_{p_1}^{p_0}\nabla_h\cdot\nabla\mathbf{u}dp=0$, we obtain from integration by parts
\beq
-\int_\mathcal M\nabla_h\Phi\cdot\mathbf{u}d\mathcal M&=&\int_\mathcal M\Phi\nabla_h\cdot\mathbf{u}d\mathcal M
  =\int_p^{p_0}\frac{ R T}{\sigma}d\sigma\nabla_h\cdot\mathbf{u}d\mathcal M\nonumber\\
  &\leq& C\|\nabla_h\mathbf{u}\|\|T\| \leq \frac{\mu_{\mathbf{u}}}{6}\|\nabla_h\mathbf{u}\|^2+C
  \|T\|^2.\label{uni.u.2}
\eeq
Noticing that $|f|\leq\frac{1}{p_0-p_1}\int_{p_1}^{p_0}|f|dp+\int_{p_1}^{p_0}|\partial_pf|dp$, it follows from Lemma \ref{lemlad} (see Appendix) that
\begin{eqnarray}
  \int_\mathcal M(\mathbf{u}\cdot\nabla_h)\mathbf{u}_2\cdot\mathbf{u}d\mathcal M
  &\leq& C\int_{\mathcal M'}\int_{p_1}^{p_0}(|\nabla_h\mathbf{u}_2|+|\nabla_h\partial_p\mathbf{u}_2|)
  dp\int_{p_1}^{p_0}|\mathbf{u}|^2d\mathcal M'\nonumber\\
  &\leq& C\|(\nabla_h\mathbf{u}_2,\nabla_h\partial_p\mathbf{u}_2)\|
  \|\mathbf{u}\|\|(\mathbf{u},\nabla_h\mathbf{u})\|\nonumber\\
  &\leq&\frac{\mu_{\mathbf{u}}}{6}\|\nabla_h\mathbf{u}\|^2
  +C(1+\|(\nabla_h\mathbf{u}_2,\nabla_h\partial_p\mathbf{u}_2)\|^2)
  \|\mathbf{u}\|^2.\label{uni.u.3}
\end{eqnarray}
Again by Lemma \ref{lemlad} and Young's inequality, we have moreover
\begin{eqnarray}
  -\int_\mathcal M\omega\partial_p\mathbf{u}_2\mathbf{u}d\mathcal M&\leq&\int_{\mathcal M'}
  \int_{p_1}^{p_0}|\nabla_h\mathbf{u}|dp\int_{p_1}^{p_0}|
  \partial_p\mathbf{u}_2||\mathbf{u}|dpd\mathcal M'\nonumber\\
  &\leq& C\|\nabla_h\mathbf{u}\|\|\partial_p\mathbf{u}_2\|^{\frac12}
  \|(\partial_p\mathbf{u_2},\nabla_h\partial_p\mathbf{u}_2)\|^{\frac12}
  \|\mathbf{u}\|^{\frac12}
  \|(\mathbf{u},\nabla_h\mathbf{u})
  \|^{\frac12} \nonumber\\
  &\leq&\frac{\mu_{\mathbf{u}}}{6}\|\nabla_h\mathbf{u}\|^2
  +C\left(1+\|\partial_p\mathbf{u}_2\|^2
  \|\left(\partial_p\mathbf{u}_2,\nabla_h\partial_p
  \mathbf{u}_2\right)\|^2\right)\|\mathbf{u}\|^2.
  \label{uni.u.4}
\end{eqnarray}
Substituting (\ref{uni.u.1})--(\ref{uni.u.4}) into (\ref{uni.u.0}) and
using the regularities of $(\mathbf{u}_i, T_i, q_{vi}, q_{ci}, q_{ri})$, we obtain
\begin{eqnarray}
  \frac{d}{dt}\|\mathbf{u}\|^2 +\mu_{\mathbf{u}}\|\nabla_h\mathbf{u}\|^2+\nu_{\mathbf{u}}
  \|\partial_p\mathbf{u}\|_w^2\leq C(1+\|
  \nabla_h\partial_p\mathbf{u}_2\|^2) \left(\|\mathbf{u}\|^2+\|T\|^2\right).\label{uni.u}
\end{eqnarray}

For the estimates of the differences for the temperature and moisture  components, we proceed in a similar fashion to \cite{HKLT} by introducing the new unknowns as
\[Q_i=q_{vi}+q_{ri}, \quad H_i=T_i-\frac{L}{c_p}(q_{ci}+q_{ri}),\]
$i=1,2$, and let $Q=Q_1-Q_2, H=H_1-H_2$ and $q_j=q_{j1}-q_{j2}$ for $j\in \{v,c,r\}$ be the corresponding differences. The source terms for these quantities reveal helpful cancellation properties allowing to prove uniqueness of $Q,q_c,q_r,H$, which then further implies the uniqueness of the solution in terms of the original unknowns  $T, q_v, q_c, q_r$.  Following the
argument of Proposition 4.1 in \cite{HKLT}, we have the estimates
\begin{eqnarray}
&&\frac12\frac{d}{dt}\|Q\|^2+\frac{\mu_{q_v}}{4}
\|\nabla_h Q\|^2+\frac{\nu_{q_{v}}}{4}\|\pa_pQ\|_w^2\nonumber \\
&&\quad\leq\int_{\mathcal M}(\mathbf{u}\cdot\nabla_h Q_2+\omega\partial_pQ_2)Qd\mathcal M+C\|(Q,q_{r},q_{c},H)\|^2
\nonumber\\
&&\qquad +C_Q \left(\mu_{q_r}\|\nabla_h  q_{r}\|^2 + \nu_{q_r}
\|\pa_p q_r\|_w^2\right),\label{est.Q.uni}\\
&&\frac12\frac{d}{dt}\|q_{c}\|^2 +\mu_{q_c}\|\nabla_h
q_{c}\|^2+\nu_{q_c}\|\pa_pq_{c}\|_w^2\nonumber\\
&&\quad\leq \int_{\mathcal M}(\mathbf{u}\cdot\nabla_h q_{c2}+\omega\partial_pq_{c2})q_cd\mathcal M+C\|(Q,q_{r},q_{c},H)\|^2,\label{est.qc.uni}\\
&&\frac12\frac{d}{dt}\|q_{r}\|^2 + \frac{\mu_{q_r}}
{4}\|\nabla_h q_{r}\|^2+\frac{\nu_{q_r}}{4}\|\pa_pq_{r}\|_w^2\nonumber \\
&&\quad\leq \int_{\mathcal M}(\mathbf{u}\cdot\nabla_h q_{r2}+\omega\partial_pq_{r2})q_rd\mathcal M+C\|(Q,q_{r},q_{c},H)\|^2,\label{est.qr.uni}\\
&&\frac12\frac{d}{dt}\|H\|^2+\frac{\mu_T}{2}
\|\nabla_hH\|^2+\frac{\nu_T}{2}\|\partial_pH\|^2_w\nonumber\\
&&\quad \leq\int_{\mathcal M}(\mathbf{u}\cdot\nabla_h H_2+\omega\partial_pH_2)Hd\mathcal M+C_H(\mu_{q_c}\|\nabla_hq_{c}\|^2+\mu_{q_r}\|\nabla_hq_{r}\|^2)\nonumber\\
&&\qquad+C_H(\nu_{q_c}\|\partial_pq_c\|^2+\nu_{q_r}\|\partial_pq_r\|^2)
+C\|(H,q_r,q_c)\|^2.\label{est.H.uni}
\end{eqnarray}
For the integrals involving the convection terms in the above inequalities, thanks to the boundary conditions (\ref{bound.0})--(\ref{bound.ll}) for $\mathbf{u}$ and $\omega$, it follows from integration by parts that
\begin{align}
  \int_\mathcal M(\mathbf{u}\cdot\nabla\varphi_2&+\omega\partial_p\varphi_2)\varphi d\mathcal M
  =-\int_\mathcal M(\mathbf{u}\cdot\nabla_h\varphi+\omega\varphi)\varphi_2d\mathcal M
  \nonumber\\
  \leq&\ \eta(\|\nabla_h\mathbf{u}\|^2+\|\nabla_h\varphi\|^2)+C_\eta
  \|\varphi_2\|_{L^\infty}^2(\|\mathbf{u}\|^2+\|\varphi\|^2)\nonumber\\
  \leq&\ \eta(\|\nabla_h\mathbf{u}\|^2+\|\nabla_h\varphi\|^2)+C_\eta
 (\|\mathbf{u}\|^2+\|\varphi\|^2),
  \label{con.uni}
\end{align}
for $\varphi\in\{Q, q_c, q_r, H\}$, and for any positive $\eta$, where
in the last step we have used the uniform boundedness of the moisture components.
Multiplying (\ref{est.qc.uni}) and (\ref{est.qr.uni}) by a sufficient large positive number $A$, adding the resultants with (\ref{est.qc.uni}) and (\ref{est.qr.uni}), and using (\ref{con.uni}), with $\eta$ sufficiently small, we obtain
\begin{eqnarray}
  \frac{d}{dt}\left(\|(Q,H)\|^2+A\|(q_r,q_c)\|^2\right)
  \leq4\eta \|\nabla_h\mathbf{u}\|^2+C_\eta \|(\mathbf{u},Q,q_{r},q_{c},H)\|^2, \label{uni.Q}
\end{eqnarray}
for a sufficiently small positive $\eta$.

Combining (\ref{uni.u}) with (\ref{uni.Q}) and choosing $\eta$ sufficiently small, one obtains
\begin{eqnarray*}
 &&\frac{d}{dt}\left(\|(Q,H)\|^2+A\|(q_r,q_c)\|^2+\|\mathbf{u}\|^2\right)\\
 &\leq& C (1+\|
  \nabla_h\partial_p\mathbf{u}_2\|^2) \|(\mathbf{u},Q,q_{r},q_{c},H)\|^2\,.
\end{eqnarray*}
Since due to the regularities of the solution we have $
 \| \nabla_h\partial_p\mathbf{u}_2\|^2\in L^1((0,\mathcal T)) $
for any $\mathcal T\in(0,\infty)$,
 the conclusion follows using Gronwall's inequality.
\end{proof}

\noindent\textbf{Acknowledgments.} S.H. acknowledges support by the Austrian Science Fund via the Hertha-Firnberg project T-764. R.K.\ acknowledges support by Deutsche Forschungsgemeinschaft through
Grant CRC 1114 ``Scaling Cascades in Complex Systems'', projects A02 and C06.
J.L. was supported in part by the National Natural Science Foundation of China grants 11971009, 11871005, and 11771156, by the Natural Science Foundation of Guangdong Province
grant 2019A1515011621, by the South China Normal University start-up grant 550-8S0315, and by the Hong Kong RGC grant CUHK 14302917. E.S.T. was supported in part by the Einstein Stiftung/Foundation - Berlin, through the Einstein Visiting Fellow Program,
and by the John Simon Guggenheim Memorial Foundation.

\section{Appendix}

\begin{lemma}\label{lem}
For any measurable function $f$ satisfying $f,\pa_p f \in L^1(\mathcal M)$ the following estimate holds
$$
\sup_{p_1\leq p\leq p_0}\|f\|_{L^1(\mathcal M')}\leq\frac{\|f\|_{L^1(\mathcal M)}}{p_0-p_1}+\|\partial_pf\|_{L^1(\mathcal M)}.
$$
\end{lemma}

\begin{proof}
 For any $p, q\in[p_1,p_0]$
\begin{equation*}
  \int_{\mathcal M'}|f(x,y,p)|d\mathcal M'=\int_{\mathcal M'}|f(x,y,q)|d\mathcal M'+\int_q^p \int_{\mathcal M'}\partial_p|f(x,y,\xi)|d\mathcal M'd\xi.
\end{equation*}
Integrating the above equality with respect to $q$ over $(p_1, p_0)$ yields
\begin{eqnarray*}
\int_{\mathcal M'}|f(x,y,p)|d\mathcal M'&=&\frac{1}{p_0-p_1}
\int_{p_1}^{p_0}\left(\int_{\mathcal M'}|f|d\mathcal M'+\int_q^p\int_{\mathcal M'}\partial_p|f(x,y,\xi)|d\mathcal M'd\xi\right)dq\nonumber\\
&\leq&\frac{\|f\|_{L^1(\mathcal M)}}{p_0-p_1}+\|\partial_pf\|_{L^1(\mathcal M)},
\end{eqnarray*}
from which, by taking the superium in $p$, the conclusion follows.
\end{proof}

We will also use the following lemma from \cite{CLT}, where we also refer to \cite{CT03,CT13} for some similar inequalities.

\begin{lemma}[See Lemma 2.1 in \cite{CLT}]
\label{lemlad}
The following inequalities hold
  \begin{eqnarray*}
    &&\int_{\mathcal M'}\left(\int_{p_1}^{p_0}|\phi|dp\right)
    \left(\int_{p_1}^{p_0}|\varphi\psi|dp\right)d\mathcal M'\\
    &\leq& C\|\phi\|\|\varphi\|^{\frac12}\left(\|\varphi\|^{\frac12}+\|\nabla_h\varphi
    \|^{\frac12}\right)
    \|\psi\|^{\frac12}\left(\|\psi\|^{\frac12}+\|\nabla_h\psi\|^{\frac12}\right),
    \end{eqnarray*}
    and
    \begin{eqnarray*}
    &&\int_{\mathcal M'}\left(\int_{p_1}^{p_0}|\phi|dp\right)\left(\int_{p_1}^{p_0}|\varphi\psi|dp\right)
    d\mathcal M'\\
    &\leq& C\|\phi\|^{\frac12}\left(\|\phi\|^{\frac12}+\|\nabla_h\phi
    \|^{\frac12}\right)\|\varphi\|^{\frac12}\left(\|\varphi\|^{\frac12}+\|\nabla_h\varphi
    \|^{\frac12}\right)\|\psi\|,
  \end{eqnarray*}
for any measurable functions $\phi,\varphi$, and $\psi$ such the quantities on the right-hand
sides are finite, where $C$ is a positive constant depending only on $p_1, p_0$, and $\mathcal M'$.
\end{lemma}

\end{document}